\documentclass[final]{siamltex}
\usepackage{showkeys}
\usepackage{amsmath,color}
\usepackage{amsfonts}
\usepackage{booktabs}
\usepackage[english]{babel}
\usepackage{graphicx}
\usepackage[opta]{optional} 
\usepackage{tikz,caption}
\usepackage{pgfplots}
\pgfplotsset{compat=newest}
\pgfplotsset{plot coordinates/math parser=false}
\usepgfplotslibrary{colormaps} 
\usetikzlibrary{pgfplots.colormaps} 
\usetikzlibrary[pgfplots.colormaps] 
\newlength\figureheight
\newlength\figurewidth 
\usetikzlibrary{external}
\usepackage[caption=false,font=footnotesize]{subfig}

\def\IR{\hbox{\rm I\kern-.2em\hbox{\rm R}}}

\newcommand{\pc}{{\cal P}_{aug}^{IPF}}
\newcommand{\pd}{{\cal P}_{aug}^{BDF}}
\newcommand{\pdr}{{\cal P}_{red}^{BDF}}
\newcommand{\pcr}{{\cal P}_{red}^{IPF}}
\newcommand{\pgen}{{\cal P}}

\newcommand{\ascp}{{\sc ssn-gmres-ipf} } 
\newcommand{\asbd}{{\sc ssn-minres-bdf} }

\newcommand{\R}{\ensuremath{\mathbb{R}}}

\newcommand{\F}{\mathcal{F}}

\newcommand{\A}{\mathcal{A}}

\newcommand{\La}{L}

\newcommand{\I}{\mathcal{I}}

\newcommand{\ltwo}{\ensuremath{L^{2}(\Omega)}}
\newcommand{\lone}{\ensuremath{L^{1}(\Omega)}}

\newcommand{\eps}{\varepsilon}

\newtheorem{teo}{Theorem}[section]

\def\minres{{\large {\sc minres}}}

\def\gmres{{\large {\sc gmres}}}

\def\norm#1{\left\|#1\right\|} 
\newcommand{\eqdef}{\stackrel{\rm def}{=}}

\usepackage{algpseudocode}
\usepackage[ruled]{algorithm}
 \renewcommand{\algorithmicrequire}{\textbf{input:}}
    \renewcommand{\algorithmicensure}{\textbf{output:}}

\begin{document}
\title{Preconditioning PDE-constrained optimization with  $\rm L^1$-sparsity and control constraints}

\author{Margherita Porcelli\footnotemark[1], Valeria Simoncini\footnotemark[2], Martin Stoll\footnotemark[3]}
\renewcommand{\thefootnote}{\fnsymbol{footnote}}
\footnotetext[1]{Universit\`a degli Studi di Firenze, 
Dipartimento di Ingegneria Industriale, 
Viale Morgagni, 40/44, 50134 Firenze, Italy (\texttt{margherita.porcelli@unifi.it})}
\footnotetext[2]{Universit\`a di Bologna,
Dipartimento di Matematica,
Piazza di Porta S.Donato 5
40127 Bologna, Italy
(\texttt{valeria.simoncini@unibo.it})
}
\footnotetext[3]{Numerical Linear Algebra for Dynamical Systems, Max Planck Institute
for Dynamics of Complex Technical Systems,
Sandtorstr. 1,
39106 Magdeburg,
Germany,
(\texttt{stollm@mpi-magdeburg.mpg.de})
}

\renewcommand{\thefootnote}{\arabic{footnote}}
\maketitle
\begin{abstract}
PDE-constrained optimization aims at finding optimal setups for partial differential equations 
so that relevant quantities are minimized. Including sparsity promoting terms in the formulation of 
such problems results in more practically relevant computed controls but adds more challenges to the numerical solution of these problems. 
The needed $\rm L^1$-terms as well as additional inclusion of box control constraints require the 
use of semismooth Newton methods.
We propose robust preconditioners for different formulations
of the Newton's equation. With the inclusion of a line-search strategy and an inexact approach for the solution of the linear systems,
the resulting semismooth Newton's method  is feasible for practical problems.  
Our results are underpinned by a theoretical analysis of the preconditioned matrix. 
Numerical experiments illustrate the robustness of the proposed scheme.
\end{abstract}

\section{Introduction}
Optimization is a crucial tool across the sciences, engineering, and life sciences and is thus requiring robust mathematical tools in terms of software and algorithms \cite{NocW06}. 
Additionally, over the last decade the need for sparse solutions has become apparent and the field of compressed sensing \cite{Don06,CanW08} is a success story where mathematical tools have
conquered all fields from image processing \cite{Nak16} to neuroscience \cite{GanS12}. 

One of the typical areas where optimization and sparsity promoting formulations are key ingredients is the optimization of functions subject to partial 
differential equation (PDE) constraints. In this field one seeks an optimal control such that the state of the system satisfies certain criteria, e.g.
being close to a desired or observed quantity, while the control and state are connected via the underlying physics of the problem. This so-called state equation is typically a PDE of a certain type. 
While this is not a new problem \cite{book::IK08,book::FT2010} the penalization of the control cost via an $\rm L^1$-norm requires different methodology from the classical $\rm L^2$-control term. 
Nonsmooth Newton methods \cite{book::IK08,HIK02,book::hpuu09,U11} have become the state-of-the-art when dealing with nonsmooth terms in PDE-constrained optimization problems as
they still exhibit local superlinear convergence. In the case of a sparsity promoting $\rm L^1$-term within PDE-constrained optimization Stadler \cite{Sta09} 
considered the applicability and convergence of the nonsmooth Newton's method. Herzog and co-authors have since followed up on this with more sophisticated sparsity structures \cite{HSW11_DS,HerOW15} 
such as directional and annular sparsity. 

The core of a Newton's scheme for a PDE constrained problem is to find a new search direction by solving a large-scale linear system in structured form, the so-called
Newton's equation. For the case of problems with a sparsity promoting term, the solution of this linear algebra phase has not yet received very much attention. 
In \cite{HS10} Herzog and Sachs consider preconditioning for optimal control problems with box control constraints
and point out that including sparsity constraints yields similar structure in the linear systems.

The aim of this work is threefold. We discuss a reduced formulation of the Newton's equation whose size does not depend on the active components,
that is components which are zero or are on the boundary of the box constraints.
We theoretically and experimentally analyze two classes of preconditioners with particular emphasis on their robustness with respect 
to the discretized problem parameters. Finally we show how line-search and inexactness in the nonlinear iterations can be exploited 
to make the Newton's method reliable and efficient.

%

The remainder of the paper is structured as follows. 
We state the main model problem based on the Poisson problem in Section~\ref{sec::modprob} and 
discuss its properties.  
Section~\ref{sec::ns} introduces the semismooth Newton's method that allows the efficient solution of the nonsmooth optimization problem. 
This scheme is based on the use of active sets representing the sparsity terms and the box control constraints and are reflected in the matrix representation 
of the generalized Jacobian matrix. We also introduce the convection-diffusion problem as a potential state equation;
this results in a nonsymmetric matrix 
representing the discretized PDE operator. We then discuss two formulations of the saddle point system 
associated with the Newton step. One of them is a full-sized system of $4\times 4$ structure while the other one is in reduced $2\times2$ form. In Section~\ref{sec::solver} the construction of efficient preconditioners is discussed and provide a thorough analysis of the preconditioners proposed for both the full and the reduced system. The bounds illustrate that the Schur-complement approximations are robust with respect to varying essential system parameters. Numerical experiments given in Section~\ref{exp} illustrate the performance of the proposed iterative schemes.

\subsection*{Notation}
The matrix $\Pi_{\mathcal{C}}$ is a diagonal (0,1) matrix with nonzero entries in the set of indices $\mathcal{C}$
and $P_{\mathcal{C}}$ is a rectangular matrix consisting of those rows of $\Pi_{\mathcal{C}}$ that belong to the
indices in ${\mathcal{C}}$  ($P^T_{\mathcal{C}}P_{\mathcal{C}} = \Pi_{\mathcal{C}})$.
Finally, given a sequence of vectors $\{x_k\}$, for any function $f$, we let $f_k = f(x_k)$.

\section{Model problems}
\label{sec::modprob}
The typical model problem in PDE-constrained optimization is usually stated as
\begin{align*}
 \F(\rm y,\rm u)=\frac{1}{2}\norm{\rm y-\rm y_d}_{\ltwo}^{2}+\frac{\alpha}{2}\norm{\rm u}^{2}_{\ltwo}
\end{align*}
where $\rm y$ is the state and $\rm u$ the control. The term $\rm y_d$ is the so-called desired state.
The state $\rm y$ and $\rm u$ are then linked via a state equation such as the Poisson equation.
The computed controls for this problem typically are `potato shaped' and distributed in space and time.
This is an unwanted feature in many applications and one commonly
 considers the model where we seek $(\rm y,\rm u) \in H_0^1 \times L^2(\Omega)$ such that the function
\begin{align}
\ \F(\rm y,\rm u)&=\frac{1}{2}\|\rm y-\rm y_d\|^ 2_{\ltwo}+ \frac{\alpha}{2}\|\rm u\|^ 2_{\ltwo} + 
\beta\|u\|_{\lone}  \label{pb}
\end{align}
is minimized subject to the constraints 
\begin{align}
-\Delta \rm y &= \rm u + \rm f  \mbox{ in } \Omega \label{eq:lap} \\        
\rm y  &= 0  \mbox{ on } \Gamma 
\end{align}
with additional box-constraints on the control
\begin{equation}\label{box}
\rm a \le \rm u \le \rm b \mbox{ a.e. in } \Omega,
\end{equation}
with $\rm y_d, f \in \ltwo$, $\rm a,b \in \ltwo$ with $\rm a<0<b$ a.e. and $\alpha, \beta >0$. The Laplace operator in (\ref{eq:lap}) could be replaced by other elliptic operators.
For the case $\beta=0$ this PDE-constrained optimization problem has been studied in great detail (see \cite{book::FT2010,book::IK08}
and the reference mentioned therein). For these problems one typically 
writes down the first order conditions, which are then discretized, and the associated equation solved.
As an alternative, one discretizes the optimization problem first and then obtain the first order conditions;
the associated nonlinear system are then solved using many well-studied algorithms.

The case $\beta>0$ is much harder as the addition of the nonsmooth sparsity term makes 
the solution of the PDE-constrained optimization problem introduced above very different to the more standard smooth optimization problems. 
The use of semismooth Newton schemes \cite{BIK99,HIK02,U11} has proven to give an efficient algorithm as well as to 
allow for a thorough analysis of the convergence behaviour. In this paper we will not discuss the conditions needed to guarantee 
the differentiability of the involved operators but mostly refer to the corresponding literature.
Here we state the following result identifying the optimality conditions of problem (\ref{pb})-(\ref{box}).
\vskip 0.1in
\begin{teo}{\rm (see \cite[Theorem 2.3]{Sta09})}
 The solution $({\rm  \bar y, \bar u})\in H_0^1(\Omega)\times L^2(\Omega)$ of the problem (\ref{pb})-(\ref{box})
 is characterized by the existence of $({\rm \bar p, \bar \xi})\in H_0^1(\Omega)\times L^2(\Omega)$
 such that
\begin{equation}\label{kkt1}
  \begin{array}{l}
-\Delta  \rm \bar  y - \rm \bar u - \rm f = 0 \\
-\Delta \rm \bar p + \bar y - y_d =0\\
- \rm \bar p + \alpha \bar u + \bar \xi = 0 \\
\rm \bar u - \max (0, \bar u + c (\bar \xi - \beta)) - \min (0, \rm \bar u + c(\bar \xi +\beta)) \\
\quad + \max(0, \rm (\bar u - b) + c (\bar \xi -\beta)) + \min (0, \rm (\bar u -a)+c(\bar \xi + \beta)) =0
\end{array}
\end{equation}
with $c>0$.
\end{teo}

\vskip 0.1in

We now consider the operator  $\rm F:L^2(\Omega)\times L^2(\Omega) \rightarrow L^2(\Omega),$ representing the nonlinear function in (\ref{kkt1}), defined by
\begin{align}
\label{f1}
\rm F(\rm u,\xi) &=  \rm u - \max (0,  u + c(\xi - \beta)) - \min (0, u + c(\xi+\beta)) \\
\nonumber &\rm \quad \quad +\max(0, ( u - b) + c(\xi -\beta)) + \min (0, (u -a)+  c(\xi + \beta)).
\end{align}
Alternatively, one could use the gradient equation 
$\rm - p + \alpha  u +  \xi = 0 $
to eliminate the Lagrange multiplier $\rm \xi$ via $ \rm \xi = p - \alpha u$ and further use $c=\alpha^{-1}$ 
obtaining the complementarity equation in the variables $\rm (u,p)$
\begin{align*}
\rm \alpha u - \max (0,   p - \beta) - \min (0,  p +\beta)+ \max(0,  p -\beta  -\alpha b)  + \min (0,  p + \beta  -\alpha a)=0,
\end{align*}
see \cite{HIK02,Sta09}.

In this work we investigate the solution of the optimality system (\ref{kkt1}):
 we first discretize the problem using finite elements (see \cite{book::hpuu09} for a more detailed discussion) and then solve the corresponding 
nonlinear system in the finite dimensional space using a semismooth Newton's method \cite{HIK02,U11}. 

\section{The semismooth Newton's method for the optimality system}
\label{sec::ns}
Let $n$ denote the dimension of the discretized space.
Let the matrix $L$ represent a discretization of the Laplacian 
operator (the \textit{stiffness matrix}) or, more generally,
be the discretization of a non-selfadjoint elliptic differential operator.
Let the matrix $M$ be the FEM Gram matrix, i.e., the so-called \textit{mass matrix}
and let the matrix $\bar M$ represent the discretization of the control term 
within the PDE-constraint. 
While for the Poisson control problem this is simply the mass matrix, i.e. $\bar M=M$, for boundary
control problems or different PDE constraints (see Section \ref{sec::CD}) the matrix 
$\bar M$ does not necessarily
coincide with $M$.  Finally, let $y,u,p,\mu,y_d, f, a,b$ be the discrete counterparts of the 
functions $\rm y,u,p,\xi,y_d, f,a,b$, respectively.

The optimality system (\ref{kkt1}) can be represented in the discretized space using 
the nonlinear function $\Theta : \IR^{4n}\rightarrow \R^{4n}$ defined as
\begin{equation}\label{Theta}
\Theta(x) =                     \begin{bmatrix}
                     \Theta(x)^y\\
                     \Theta(x)^u \\
                     \Theta(x)^p \\
                     \Theta(x)^{\mu}
                    \end{bmatrix} \eqdef \begin{bmatrix}
                     M y +L^T p - M y_d\\
                     \alpha M u - \bar M^T p + M\mu \\
                     Ly - \bar M u - f \\
                     M F(u,\mu)
                    \end{bmatrix}                    
\end{equation}
where $x= (y,u,p,\mu)\in \IR^{4n}$ and the 
discretized complementarity function $F:\IR^{2n}\rightarrow \IR^{2n}$ is component-wise defined by
\begin{align*}
 F( u,\mu) &=   u - \max (0,  u + c(\mu - \beta)) - \min (0, u + c(\mu +\beta)) \\
\nonumber & \quad \quad +\max(0, ( u - b) + c(\mu -\beta)) + \min (0, (u -a)+  c(\mu + \beta)).
\end{align*}
for $c>0$.

Due to the presence of the min/max functions in the complementarity function $F$, 
function $\Theta$ is {\em semismooth} \cite{book::hpuu09,U11}. The natural extension of the classical Newton's method 
is
\begin{equation}\label{ssn}
x_{k+1} = x_k -(\Theta'(x_k))^{-1} \Theta(x_k), \quad \mbox{ for } k = 0,1, \dots
\end{equation}
where $\Theta'(x_k)\in \IR^{4x \times 4n}$ is the generalized Jacobian of $\Theta$ at $x_k$ \cite{U11}.
Under suitable standard assumptions, the generalized Jacobian based Newton's method (\ref{ssn}) converges superlinearly \cite{U11}.
We employ an extension of the above method proposed in \cite{MQ95} and reported in Algorithm \ref{algo}
that takes into account both the inexact solution of the Newton's equation and 
the use of a globalization strategy based on the merit function
$$
\theta(x) = \frac 1 2 \|\Theta(x)\|_2^2.
$$
At Lines 3-4 the Newton's equation is solved with an accuracy  controlled by the forcing term $\eta_k >0$; 
Lines 6-8 consist in a line-search where the sufficient decrease is measured with respect to 
the (nonsmooth) merit function $\theta$.

Under certain conditions on the sequence $\eta_k$ in (\ref{res}), the method retains the superlinear local convergence 
of (\ref{ssn}) and gains the convergence to a solution of the nonlinear system starting  from any $x_0$ \cite{MQ95}.

\begin{algorithm}
\caption{Global Inexact Semismooth Newton's method \cite{MQ95}}\label{algo}
\begin{algorithmic}[1]
\Require Starting $x_0$ and $\tau_0$, parameters $\sigma \in (0,1]$, $\gamma\in (0,1)$, $\tau_{\theta}>0$.
\While{$\theta(x_k) > \tau_{\theta}$}
\State Choose $\eta_k\in (0,1)$.
\State Solve 
\begin{equation}\label{NEQ}
 \Theta'(x_k) \Delta x = - \Theta(x_k) + r_k,
\end{equation}
\State with
\begin{equation} \label{res}
   \|r_k\|_2 \le \eta_k \|\Theta(x_k)\|_2.
\end{equation}
\State $\rho \gets 1$; 
\While{$\theta(x_k + \rho \Delta x) - \theta(x_k) > - 2 \sigma \gamma \rho \theta(x_k)$}
\State $\rho \gets \rho/2$.
\EndWhile
\State $x_{k+1} \gets x_k + \rho \Delta x$, $k \gets k+1$. 
\EndWhile
\end{algorithmic}
\end{algorithm}

The form of $\Theta'(x)$ can be easily derived using an active-set approach as follows. 
Let us define the following sets
\begin{eqnarray}
{\A_b} & = & \{ i \ | \ c(\mu_i -\beta) + (u_i-b_i)>0 \} \nonumber \\
{\A_a} & = & \{ i \ | \ c(\mu_i + \beta) +(a_i-u_i)<0 \} \nonumber \\
{\A_0} & = & \{ i \ | \  u_i + c( \mu_i + \beta) \ge 0 \} \cup \{ i \ | \   u_i + c(\mu_i - \beta) \le 0 \}    \label{A0} \\ 
{\mathcal{I}_+} & = & \{ i \ | \   u_i + c(\mu_i - \beta) >0     \}\cup  i \ | \  c(\mu_i -\beta) + (u_i-b_i) \le 0  \} \label{Ip} \\
{\mathcal{I}_-} & = & \{ i \ | \   u_i + c(\mu_i + \beta) < 0 \} \cup  \{ i \ | \ c(\mu_i +\beta) + (u_i-a_i) \ge 0 \}. \label{Im} 
\end{eqnarray}
Note that the above five sets are disjoint and if  
$${\A} \eqdef  {\A_b} \cup {\A_a} \cup {\A_0} $$
is the set of {\em active constraints} then its complementary set of {\em inactive}
constraints is 
$${\mathcal{I}} \eqdef {\mathcal{I}_+} \cup {\mathcal{I}_-}.$$

With these definitions at hand, the complementarity function $F$ can be expressed in compact form as
\begin{equation}\label{Fdis}
 F(u,\mu) =  \Pi_{\A_0} u +\Pi_{\A_b} (u-b) + \Pi_{\A_a} (u-a)-
c(\Pi_{\mathcal{I}_+}(\mu - \beta) + \Pi_{\mathcal{I}_-}(\mu + \beta)).
\end{equation}
It follows from the complementarity conditions that 
\begin{itemize}
 \item $u_i = 0$ for $i\in {\A_0}$;
 \item $u_i = a$ for $i\in {\A_a}$ and $u_i = b$ for $i\in {\A_b}$;
 \item $\mu_i = - \beta$ for $i\in {\mathcal{I}_-}$ and $\mu_i = \beta$ for $i\in {\mathcal{I}_+}$.
\end{itemize}
From (\ref{Fdis}),  the (generalized) derivative of $F$ follows, that is
$$
F'(u,\mu) = 
\begin{bmatrix}
 \Pi_{\A}  & - c\Pi_{\mathcal{I}}
\end{bmatrix},
$$
together with the corresponding Jacobian matrix of $\Theta$,
$$
\Theta'(x) = 
\begin{bmatrix}
 M & 0 & \La^T & 0 \\
0 & \alpha M & -\bar M^T & M  \\
\La & - \bar M & 0 & 0 \\        
0 &  \Pi_{\A}M & 0 & -c\Pi_{\mathcal{I}} M
\end{bmatrix}.
$$
We note that $\Theta'(x)$ depends on the variable $x$ through the definition of the sets $\A$ and $\I$.


\subsection{Other PDE constraints}
\label{sec::CD}
It is clear that the above discussion is not limited to the Poisson problem in
(\ref{eq:lap}) but can also be extended to different models. We consider the  convection-diffusion equation 
\begin{align}
\label{eq::CD}
-\eps \triangle \rm y +w\cdot \nabla y &= \rm u \textrm{ in } \Omega \\
\rm y(:,x) &= \rm f \textrm{ on }\Gamma\\
\rm y(0,:)&= \rm y_0
\end{align}
as constraint to the objective function (\ref{pb}).
The parameter $\eps$ is crucial to the convection-diffusion equation
as a decrease in its value makes the equation more convection dominated. The wind $\rm w$
is predefined. Such optimization problems have been recently analyzed in
\cite{Ree10,HeiL10,PW11,axelsson2015comparison} and we refer to these references
for the possible pitfalls regarding the discretization. We focus on a 
discretize-then-optimize scheme using a streamline upwind Petrov-Galerkin (SUPG)
approach introduced in \cite{BroH82}. Note that other schemes such as
discontinuous Galerkin methods \cite{Sun10} or local projection stabilization
\cite{PW11} may be more suitable discretizations for the optimal
control setup as they often provide the commutation between optimize first or
discretize first for the first order conditions. Nevertheless, our approach will
also work for these discretizations. We employ linear finite elements with an SUPG 
stabilization that accounts for the convective term. The discretization of the PDE-constraint is now different as we obtain 
\begin{align}
\La y-\bar M u= f
\end{align}
with $\La$ the SUPG discretization of the differential operator in \eqref{eq::CD} 
(see \cite{ElmSW14,PW11} for more details). The matrix $\bar M$ 
now includes an extra term that takes into account the SUPG correction. Entry-wise this matrix is given as
$$
(\bar M)_{ij}=\int_\Omega \phi_i\phi_j+\delta
\int_\Omega\phi_i\left({\rm w}\cdot\nabla\phi_j\right) ,
$$
where $\{\phi_j\}$ are the finite element test functions and $\delta$ is a parameter
coming from the use of SUPG \cite{BroH82,ElmSW14}. 
The resulting matrices are both unsymmetric and hence forward problems require the use of nonsymmetric iterative solvers.
{While the optimality system still remains symmetric the nonsymmetric operators have to be approximated as part of the Schur-complement approximation and require more attention than the simpler Laplacian.}

\subsection{Solving the Newton's equation ``exactly''}\label{sec::exact}
Let us assume to solve the Newton's equation (\ref{NEQ}) ``exactly'', that is
to use $\eta_k = 0$ in (\ref{res}) in Algorithm \ref{algo}.

Given the current iterate  $x_k = (y_k,u_k,p_k,\mu_k)$ and the current
active and inactive sets  $\A_k$ and $\I_k$, a step of the semismooth Newton's method applied 
to the system $\Theta(x)=0$ with $\Theta$ in (\ref{Theta}) has the form
\begin{equation}\label{fulleqN1}
\begin{bmatrix}
 M & 0 & \La^T & 0 \\
0 & \alpha M & -\bar M^T & M  \\
\La & - \bar M & 0 & 0 \\        
0 &  \Pi_{\A_k}M & 0 & -c\Pi_{\mathcal{I}_k} M
\end{bmatrix}
\begin{bmatrix}
 \Delta y \\
\Delta u  \\
\Delta p \\        
\Delta \mu
\end{bmatrix}
 =
 -       \begin{bmatrix}
                     \Theta_k^y\\
                     \Theta_k^u \\
                     \Theta_k^p \\
                     \Theta_k^{\mu}
                    \end{bmatrix}.
\end{equation}

We note that the system \eqref{fulleqN1} is nonsymmetric, which would require the use of 
nonsymmetric iterative solvers. The problem can be symmetrized, so that 
better understood and cheaper symmetric iterative methods
than nonsymmetric ones  may be used for its solution.
More precisely, the system in
(\ref{fulleqN1}) can be symmetrized by eliminating $(\mu_k + \Delta \mu)_{\I_k}$
from the last row and obtaining 
\begin{equation}\label{eqN_4}
  \begin{bmatrix}
    M & 0  & \La^T & 0 \\
0 & \alpha M & -\bar M^T & M P^T_{\A_k}  \\
\La & - \bar M & 0 & 0 \\        
0 &  P_{\A_k} M & 0 &      
        \end{bmatrix}
        \begin{bmatrix}
         \Delta y \\
\Delta u  \\
\Delta p \\        
(\Delta \mu)_{\A_k} 
        \end{bmatrix}
=  -\begin{bmatrix}
    \Theta_k^y\\
 \Theta_k^u + \Gamma^u_k   \\
 \Theta_k^p \\        
 \Theta_k^{\mu} + \Gamma^{\mu}_k
   \end{bmatrix} ,
\end{equation}
with 
\begin{equation}\label{mubeta}
 (\mu_{k+1})_{(\mathcal{I_+})_{k+1}} = \beta \mbox{ \ and \ } (\mu_{k+1})_{(\mathcal{I_-})_{k+1}} = - \beta ,
\end{equation}
while
$$
\Gamma_k^u  = M \Pi_{\I_k}(\mu_{k+1}-\mu_{k})    \mbox{ \ and \ } \Gamma_k^{\mu} =  P_{\A_k} M F(u_k,\mu_k) .
$$

The dimension of the system (\ref{eqN_4}) is $(3n+n_{\A_k})\times (3n+n_{\A_k})$  and therefore 
depends on the size $n_{\A_k}$ of the active set. Since we expect $n_{\A_k}$ to be large if the optimal control 
is very sparse, that is when $\beta$ is large, we now derive a symmetric reduced system 
whose dimension is independent of the active-set strategy and that shares the same properties of
the system above.

First, we reduce the variable $\Delta u$
\begin{equation}\label{deltau}
\Delta u = \frac 1 {\alpha} (M^{-1}\bar M^T \Delta p - 
P_{\A_k}^T (\Delta\mu)_{\A_k} - M^{-1}(\Theta_k^u+\Gamma_k^u)), 
\end{equation}
yielding
\begin{equation*}
\begin{bmatrix}
M  &\La^T  & 0 \\
\La & - \frac 1 {\alpha} \bar M M^{-1}\bar M^T &  \frac 1 {\alpha} \bar M P_{\A_k}^T \\
0 & \frac 1 {\alpha} P_{\A_k} \bar M^T &  - \frac 1 {\alpha} P_{\A_k} MP_{\A_k}^T  
\end{bmatrix}
\begin{bmatrix}
\Delta y \\ \Delta p \\ (\Delta\mu)_{\A_k}
\end{bmatrix}
= -
\begin{bmatrix}
\Theta_k^y \\
\frac 1 {\alpha} \bar M M^{-1}(\Theta_k^u+\Gamma_k^u) + \Theta_k^p   \\
 \Gamma_k^{\mu} - \frac 1 {\alpha} P_{\A_k} ( \Theta_k^u +\Gamma_k^u) 
\end{bmatrix},
\end{equation*}
that is still a symmetric saddle-point system. 
Then, since $P_{\A_k} M P_{\A_k}^T$ is nonsingular, we can reduce further and
get
\begin{equation}\label{eqN_2}
\begin{bmatrix}
M  & \La^T   \\
\La & - \frac 1 {\alpha} \bar M M^{-1}\Pi_{\I_k}\bar M^T  
\end{bmatrix}
\begin{bmatrix}
\Delta y \\ \Delta p 
\end{bmatrix}
= -
\begin{bmatrix}
\Theta_k^y \\ 
\Theta_k^p+\bar M M^{-1}
\left(\Pi_{\A}  P_{\A}^T\Theta_k^{\mu}  - \frac 1 \alpha  \Pi_{\I} ( \Theta_k^u +\Gamma_k^u) \right)
\end{bmatrix}
\end{equation}
together with
\begin{equation}\label{deltamuA}
 (\Delta \mu)_{\A_k} = P_{\A_k}M^{-1}\Pi_{\A_k}\bar M^T\Delta p
- \alpha P_{\A_k}M^{-1}  P_{\A_k}^T 
\left (\Theta_k^{\mu} + \frac 1 {\alpha} P_{\A_k} (\Theta_k^u +\Gamma_k^u) \right),
\end{equation}
and $\Delta u$ in (\ref{deltau}).
We note that the dimension of system (\ref{eqN_2}) is now $2n \times 2n$ and that the computation of $(\Delta \mu)_{\A_k} $ and $\Delta u$ only involves the inversion of
the diagonal matrix $M$. Moreover, the (2,2) matrix term does not need not be formed explicitly.

Finally, we remark that if the initial approximation is ``feasible'', that is it solves the linear equations
$\Theta^y(x_0)=\Theta^u(x_0)=\Theta^p(x_0)=0$, then the residuals $\Theta_k^y=\Theta_k^u=\Theta_k^p$ 
remain zero for all $k>0$ and therefore the expressions (\ref{fulleqN1})-(\ref{deltamuA}) simplify.

In the following sections, we refer to (\ref{eqN_4}) and (\ref{eqN_2}) as the {\em augmented}  and
{\em reduced} system, respectively and denote the corresponding systems as
$$J^{aug}_k \Delta x^{aug} = b^{aug}_k \qquad \Leftrightarrow \qquad \mbox{ eq. } (\ref{eqN_4})$$
and
$$J^{red}_k \Delta x^{red} = b^{red}_k \qquad \Leftrightarrow \qquad \mbox{ eq. } (\ref{eqN_2}) . $$

\subsection{Solving the Newton's equation ``inexactly''}\label{sec::inexact}
We derive suitable inexact conditions on the residual norm of the systems
(\ref{eqN_4}) and (\ref{eqN_2}) in order to recover the local convergence 
properties of Algorithm \ref{algo} and, at the same time, exploit the symmetry
of the linear systems.

Let us partition the residual $r_k = \Theta'_k \Delta x + \Theta_k $ of the 
linear system (\ref{fulleqN1}) as
$r_k = (r_k^y, r_k^u, r_k^p,r_k^{\mu})$ and assume that 
$(r_k^{\mu})_{\I_k} = 0$. This simplification allows the substitution (\ref{mubeta}). 

Let $\tilde r_k = (r_k^y, r_k^u, r_k^p,(r_k^{\mu})_{\A_k})$.
Then, the steps (\ref{NEQ}) and (\ref{res}) of Algorithm \ref{algo} correspond to solve the augmented system (\ref{eqN_4}) as follows
\begin{equation}\label{eqN_4_inex}
J_k^{aug} \Delta x^{aug} = b_k^{aug} + \tilde r_k \mbox{ \ with \ } \|\tilde r_k \| \le \eta_k \|\Theta(x_k)\|.
\end{equation}
Moreover, let $r_k^{red}$ be the residual in the reduced system (\ref{eqN_2}). We next show that
$\|r_k^{red}\| = \|\tilde r_k\| = \|r_k\|$, so that we can solve the
reduced Newton's equation inexactly by imposing the variable accuracy explicitly on the reduced residual,
instead of imposing it on the original residual. 
More precisely, we find $\Delta x^{red}$ with residual $r_k^{red}$ such
that
\begin{equation}\label{eqN_2_inex}
 J_k^{red} \Delta x^{red} = b_k^{red} +  r_k^{red} \mbox{ \ with \ } \|r_k^{red}\| \le \eta_k \|\Theta(x_k)\|.
\end{equation}
After that, we can  recover $\Delta u$ and $(\Delta \mu)_{\A_k}$ from (\ref{deltau}) and (\ref{deltamuA}), 
respectively.

The norm equality result is very general, as it holds for any reduced system. We thus introduce
 a more general notation.
Let us consider the block linear system
\begin{eqnarray}\label{eqn:block}
{\cal K} x = b \quad \Leftrightarrow \quad
 \begin{bmatrix}
-K & G^T \\ G & C
\end{bmatrix}
\begin{bmatrix}
x_1 \\ x_2
\end{bmatrix}
=
\begin{bmatrix}
b_1 \\ b_2
\end{bmatrix} ,
\end{eqnarray}
with $C$ nonsingular.
The Schur complement system associated with the first block is
given by $(-K-G^TC^{-1}G) x_1 = \hat b_1$, with $\hat b_1=b_1-G^TC^{-1}b_2$.  The following result holds.

\begin{proposition}
Let $r^{red}$ be the residual obtained by approximately
solving the reduced (Schur complement) system
$(-K-G^TC^{-1}G) x_1 = \hat b_1$, with $\hat b_1=b_1-G^TC^{-1}b_2$. Then the residual
$\tilde r = {\cal K} x - b$ satisfies $\|\tilde r\| = \|r^{red}\|$.
\end{proposition}
\begin{proof}
We write
$$
 \begin{bmatrix}
-K & G^T \\ G & C
\end{bmatrix}
=
 \begin{bmatrix}
I & G^T C^{-1} \\ 0 & I
\end{bmatrix}
 \begin{bmatrix}
-K-G^TC^{-1}G & 0 \\ 0 & C
\end{bmatrix}
 \begin{bmatrix}
I & 0 \\ G^TC^{-1}  & I
\end{bmatrix} =: 
{\cal U} {\cal D} {\cal U}^T.
$$
Solving ${\cal K}x=b$ by reduction of the second block corresponds to
using the factorization above as follows. Setting $\widehat x = {\cal U}^T u$ we have
$$
{\cal U} {\cal D} {\cal U}^T x = b \,\, \Leftrightarrow \,\,
 {\cal D} \widehat x = {\cal U}^{-1} b \,\, \Leftrightarrow \,\,
(-K-G^TC^{-1}G) \widehat x_1 = ({\cal U}^{-1} b)_1, 
C \widehat x_2 = ({\cal U}^{-1} b)_2. 
$$
Hence, $x_1=\widehat x_1$ and $x_2 = \widehat x_2 + C^{-1} G x_1$. Let
$\widetilde x_1$ be an approximation to $\widehat x_1$, so that
$(-K-G^TC^{-1}G) \widetilde x_1 = ({\cal U}^{-1} b)_1 + r^{red}$ for some residual vector
$r^{red}$. Let $\widetilde x_2$
be defined consequently, so that $\widetilde x = [\widetilde x_1;  \widetilde x_2]$.
Substituting, we obtain
$$
\tilde r = {\cal U} ({\cal D} {\cal U}^T \widetilde x - {\cal U}^{-1}b) = [r^{red} ; 0],
$$
from which the result follows.
\end{proof}

This result can be applied to our  2$\times$2
reduced system after row and column permutation of the original 4$\times$4 system, so
that the variables used in the reduced system appear first.

\section{Iterative solution and preconditioning}
\label{sec::solver}
We now discuss the solution of the linear systems (\ref{eqN_4}) and (\ref{eqN_2})  presented earlier. 
For the ease of the presentation, we omit in this section the subscript $k$.

While direct solvers impress with performance for two-dimensional problems and moderate sized 
three-dimensional one they often run out of steam when dealing with more structured or 
general three-dimensional problems.
In this case one resorts to iterative solvers, typically Krylov subspace solvers that approximate the solution
in a Krylov subspace
$$
\mathcal{K}_{\ell}=\left\lbrace \pgen^{-1}r_0,\pgen^{-1}\mathcal{J}r_0,\left(\pgen^{-1}\mathcal{J}\right)^2r_0,
\ldots,\left(\pgen^{-1}\mathcal{J}\right)^{l-1}r_0\right\rbrace
$$
where $r_0=b-\mathcal{J}x_0$ is the initial residual. The matrix $\pgen$ is the preconditioner
that approximates $\mathcal{J}$ in some sense, and it is cheap to apply \cite{book::saad,ElmSW14,BenGolLie05}. 
In the context of PDE problems $\pgen$ is often derived from the representation 
of the inner products of the underlying function spaces \cite{MW10,GunHS,SchU12}. 
The construction of appropriate preconditioners follows the strategy to approximate the leading block of the saddle point system and correspondingly the Schur complement of this matrix. 
As iterative solvers is concerned, the indefiniteness of the symmetric system calls for
the use of a short-term algorithm such as
\minres\ \cite{minres}. To maintain its properties, the accompanying preconditioner 
should be symmetric and positive definite.
In case when the system matrix is nonsymmetric or the preconditioner is indefinite,
many competing methods that are either based on the Arnoldi process or the 
nonsymmetric Lanczos method can be applied \cite{book::saad}. 
As a rule of thumb, once a good preconditioner is constructed the choice of the nonsymmetric solver 
becomes less important. We now discuss the construction of preconditioners in more detail.

We consider preconditioning techniques based on the {\em active-set Schur complement} approximations proposed in \cite{pst15} and 
tailor this idea to the combination of sparsity terms and box control constraint case.  Our strategy here uses a matching technique \cite{PW10}
that allows the parameter-robust approximation of the Schur-complement while still being practically useful. This technique was already used for state-constraints \cite{SPW10}
not including sparsity terms and has previously shown to be successful for active set methods \cite{pst15}.

Taking the submatrix $\begin{bmatrix}
  M & 0 \\ 0 & \alpha M
 \end{bmatrix}
$ as the $(1,1)$ block, the active-set Schur complement of the matrix 
$J^{aug}$ in (\ref{eqN_4}) in its factorized form is defined as
\begin{eqnarray*} 
S & = & \frac{1}{\alpha}\begin{bmatrix}
\alpha \La M^{-1} \La^T + \bar M M^{-1} \bar M^T & - \bar M P_{\A}^T  \\
 - P_{\A} \bar M^T  & P_{\A} M P_{\A}^T 
\end{bmatrix} \\
&=& \frac 1 \alpha \begin{bmatrix} I & -\bar M \Pi_{\A} M^{-1}P_{\A}^T
 \\ 0 & I 
\end{bmatrix} 
 \begin{bmatrix}
{\mathbb S}  & 0 \\ 0 & 
P_{\A} M P_{\A}^T 
\end{bmatrix}  
  \begin{bmatrix} I & -\bar M \Pi_{\A} M^{-1}P_{\A}^T
 \\ 0 & I 
\end{bmatrix}^T ,
\end{eqnarray*} 
with
\begin{equation}\label{schurF1}
 {\mathbb S} =   \alpha \La M^{-1} \La^T + \bar M \, \Pi_{\mathcal{I}} \,M^{-1} \bar M^T.
\end{equation}
The matrix ${\mathbb S}$, and thus $S$,  explicitly depends on the current Newton's 
iteration  through the change in the active and inactive sets.

We consider the following {\em factorized approximation} of ${\mathbb S}$:
\begin{equation}\label{smhat}
\widehat {\mathbb S} : = (\sqrt{\alpha}\La +\bar M \,\Pi_{\mathcal{I}} ) M^{-1} 
(\sqrt{\alpha}\La +\bar M\, \Pi_{\mathcal{I}} ) ^T
\end{equation}
which extends the approximation proposed in \cite{pst15} to the case $M\neq \bar M$.
The approximation $\widehat {\mathbb S}$ yields the factorized active-set Schur complement approximation
\begin{equation} \label{schurmS}
\widehat S =  \frac 1 \alpha \begin{bmatrix} I & -\bar M \Pi_{\A} M^{-1}P_{\A}^T
 \\ 0 & I 
\end{bmatrix} 
 \begin{bmatrix}
\widehat  {\mathbb S}  & 0 \\ 0 & 
P_{\A} M P_{\A}^T 
\end{bmatrix}  
  \begin{bmatrix} I & -\bar M \Pi_{\A} M^{-1}P_{\A}^T
 \\ 0 & I 
\end{bmatrix}^T.
\end{equation}

The following results are along the lines of the analysis conducted in \cite{pst15}
considering $\bar M \neq M$ and $\bar M$ nonsymmetric. As in \cite{pst15}, we assume that $M$ is diagonal 
so that $M$ commutes with both $\Pi_{\cal I}$ and $\Pi_{\cal A}$; for alternatives we refer
the reader to the discussion in \cite[Remark 4.1]{pst15}.
\vskip 0.1in
\begin{proposition}\label{prop:hatS=SpG_gen}
Let ${\mathbb S}$ and $\widehat {\mathbb S}$ be as defined above. Then
$$
 \widehat{\mathbb S}  =  {\mathbb S} + \sqrt{\alpha}( \La M^{-1}\, \Pi_{\mathcal{I}}\,\bar M^T 
 +\bar M \,\Pi_{\mathcal{I}}\, M^{-1} \La^T) .
$$
\end{proposition}
{\it Proof.} The result follows from
\begin{eqnarray*}
\widehat{\mathbb S} & = & (\sqrt{\alpha}\La M^{-1} + \bar M \Pi_{\I} M^{-1}) 
(\sqrt{\alpha}\La + \bar M \Pi_{\I})^T \\
& = & \alpha  \La M^{-1} \La^T + \bar M \Pi_{\I} M^{-1} \bar M^T + \sqrt{\alpha}\La M^{-1}\Pi_{\I} \bar M ^T 
+ \sqrt{\alpha} \bar M \Pi_{\I} M^{-1} \La^T  \\
& = & {\mathbb S} + \sqrt{\alpha}( \La M^{-1}\, \Pi_{\mathcal{I}}\,\bar M^T 
 +\bar M \,\Pi_{\mathcal{I}}\, M^{-1} \La^T). \qquad \endproof
\end{eqnarray*}

If $\Pi_{\cal I} = 0$, that is all indices are active, then the two matrices coincide. 
In general,
$$
 I - \widehat{\mathbb S}^{-1}{\mathbb S} =
\widehat{\mathbb S}^{-1} \sqrt{\alpha}( \La M^{-1}\, \Pi_{\mathcal{I}}\,\bar M^T 
 +\bar M \,\Pi_{\mathcal{I}}\, M^{-1} \La^T) ,
$$
and the right-hand side matrix is of low rank, with a rank that is at most twice
the number of inactive indices. In other words, $\widehat{\mathbb S}^{-1}{\mathbb S}$
has at least a number of unit eigenvalues corresponding to half the
number of active indeces.

In the unconstrained case ($\beta =0$ and no bound constraints) and for 
$\bar M = M$, the approximation in (\ref{smhat}) corresponds to 
the approximation proposed in \cite{PW10, PW11}.

 We now derive general estimates for the
inclusion interval for the eigenvalues
of the pencil $({\mathbb S}, \widehat{\mathbb S})$, whose
extremes depend on the spectral properties of the nonsymmetric matrices $\La$ and $\bar M$ and of $M$.
\vskip 0.1in
\begin{proposition}\label{prop:upperestimate}
Let $\lambda$ be an eigenvalue of $\widehat {\mathbb S}^{-1} {\mathbb S}$.
Then it holds
$$
\frac{1}{2} \le \lambda \le \zeta^2 + (1 + \zeta)^2 ,
$$
with
$$
\zeta = \|M^{\frac 1 2} \left (\sqrt{\alpha} \La  + \bar M \Pi_{\I} \right )^{-1} \sqrt{\alpha} \La M^{-\frac 1 2} \|.
$$
Moreover, if $\La\bar M^T+\bar M\La^T\succ 0$, then
for $\alpha\to 0$, $\zeta$ is bounded by a constant independent of $\alpha$.
\end{proposition}
\vskip 0.1in
\begin{proof}
Let $F \eqdef \sqrt{\alpha} M^{-\frac 1 2} \La M^{-\frac 1 2}$ and $N\eqdef M^{-\frac 1 2}\bar M \Pi_{\I}M^{-\frac 1 2}$. Then we have
$$  M^{-\frac 1 2}{\mathbb S} M^{-\frac 1 2}  = FF^T + N N^T$$
and
$$M^{-\frac 1 2}\widehat {\mathbb S}M^{-\frac 1 2} =  (F+N)(F+N)^T.$$
We first provide the lower bound $\lambda \ge \frac 1 2$.
Let $W \eqdef F^{-1}N$.
For $x\ne 0$ we can write
$$
\lambda =
\frac{ x^T {\mathbb S}x}{x^T \widehat{\mathbb S} x} = 
\frac{ z^T (FF^T + N N^T )z}{z^T  (F+N)(F+N)^T  z} =
\frac{ y^T (WW^T + I )y}{y^T (W+I)(W+I)^T  y},
$$
where $z = M ^ {\frac 1 2 }x$ and $y= F^T z$. Then 
$\lambda \ge \frac 1 2$ if and only
if 
$$
\frac{ y^T (WW^T + I )y}{y^T (W+I)(W+I)^T  y} \ge \frac 1 2 
$$
which holds since $ 2(WW^T + I) - (W+I)(W+I)^T= (W-I)(W-I)^T \succ  0$.

We now prove the upper bound.
From $M^{-\frac 1 2} {\mathbb S} M^{-\frac 1 2}  x = \lambda M^{-\frac 1 2} \widehat {\mathbb S} M^{-\frac 1 2} x$
we obtain for $z=M^{\frac 1 2}x$
\begin{eqnarray}\label{eqn:Feig}
(F+N)^{-1} (FF^T + NN^T) (F+N)^{-T} z = \lambda z .
\end{eqnarray}
Recalling the definition of $W$ we have
\begin{eqnarray*}
(F+N )^{-1} (FF^T + NN^T) (F+N)^{-T} =
(I+W)^{-1} (I+WW^T)(I+W)^{-T}
\end{eqnarray*}
Therefore, from (\ref{eqn:Feig}) it follows
\begin{eqnarray}
\lambda  &\le& 
 \|(I+W)^{-1} (I+WW^T)(I+W)^{-T}\|  
\le  \|(I+W)^{-1}\|^2 + \|(I+W)^{-1}W\|^2  \nonumber\\
&=& \|(I+W)^{-1}\|^2 + \|I- (I+W)^{-1}\|^2 \nonumber \\
&\le& \|(I+W)^{-1}\|^2 + (1 + \|(I+W)^{-1}\|)^2.\label{eqn:lambda}
\end{eqnarray}
Recalling that 
$$
W = F^{-1}N = 
 \frac 1 {\sqrt{\alpha}} M^{\frac 1 2} \La^{-1} M^{\frac 1 2} N =
\frac 1 {\sqrt{\alpha}} M^{\frac 1 2}  \La^{-1}\bar M \Pi_{\I} M^{-\frac 1 2} =
\frac 1 {\sqrt{\alpha}} M^{\frac 1 2}  \La^{-1}\bar M M^{-\frac 1 2} \Pi_{\I}.
$$
we have 
\begin{eqnarray*}
\zeta = \|(I+W)^{-1}\| &=& \|(I+ \frac {1} {\sqrt{\alpha}} M^{\frac 1 2} \La^{-1} \bar M \Pi_{\I} M^{-\frac 1 2})^{-1} \| \\
& = &  \|M^{\frac 1 2} \left (\sqrt{\alpha} \La  + \bar M \Pi_{\I}\right )^{-1} \sqrt{\alpha} \La M^{-\frac 1 2} \|.
\end{eqnarray*}
%
%
%
To analyze the behavior for $\alpha\to 0$, let us suppose that $L\bar M^T +\bar M L^T \succ 0$.
Without loss of generality assume that $\Pi_\I = {\rm blkdiag}(I_\ell,0)$.
Let $Z=M^{\frac 1 2}  \La^{-1}\bar M M^{-\frac 1 2}$, so that
$W = \frac 1 {\sqrt{\alpha}} Z \Pi_\I = \frac 1 {\sqrt{\alpha}} [Z_{11},{Z_{12}};0,0]$.
Thanks to the hypothesis on $\bar M L^T$, the matrix $Z_{11}$ is also positive definite, that
is its eigenvalues all have positive real part. Let 
$Z\Pi_\I = X {\rm blkdiag}(\Lambda_1,0) X^{-1}$ be the eigendecomposition \footnote{In the 
unlikely case of a Jordan decomposition, the proof proceeds with the maximum over norms of 
Jordan blocks inverses, which leads to the same final result.} of $Z\Pi_\I$. Then
%
{
\begin{eqnarray*}
\|(I+W)^{-1}\| 
\le
{\rm cond}(X) 
\max \left\{ \frac 1 {\displaystyle\min_{\lambda\in {\rm spec}(Z_{11})} |1+\lambda/\sqrt{\alpha}|}, 1\right\} .
\end{eqnarray*}
}
We thus have
$$
\max \left\{ \frac 1 {\displaystyle \min_{\lambda\in {\rm spec}( Z_{11})} |1+\lambda/\sqrt{\alpha}|}, 1\right\} \to 1 
\qquad {\rm for } \qquad \alpha \to 0 ,
$$
so that $\|(I+W)^{-1}\| \le \eta\, {\rm cond}(X)$ with $\eta\to 1$ for $\alpha \to 0$.
\end{proof}

\vskip 0.1in

In practice, the matrix $\widehat {\mathbb S}$ is replaced by an approximation
whose inverse is cheaper to apply. This is commonly performed by using a spectrally 
equivalent matrix $\check {\mathbb S}$, so that there exist two positive constants $c_1, c_2$
independent of the mesh parameter such that
$c_1 x^T {\check{\mathbb S}} x  \le
 x^T {\widehat{\mathbb S}} x  \le
c_2 x^T {\check{\mathbb S}} x$.
For this choice, we thus obtain the following spectral bounds for $\check{\mathbb S}^{-1}{\mathbb S}$,
$$
 \frac {c_1}{2}\le \frac{x^T {\mathbb S} x}{x^T {\check{\mathbb S}} x}  =
 \frac{x^T {\mathbb S} x}{ x^T {\widehat{\mathbb S}} x}  
 \frac{x^T \widehat {\mathbb S} x}{ x^T {\check{\mathbb S}} x}  \le 
c_2 (\zeta^2+(1+\zeta)^2) .
$$

Based on $\widehat {\mathbb S}$ and $\widehat{S}$ defined in (\ref{smhat}) and (\ref{schurmS}), respectively, 
we introduce convenient preconditioners for the linear systems (\ref{eqN_4}) and (\ref{eqN_2}).
For the augmented linear system in (\ref{eqN_4}) we consider a block diagonal preconditioner
and  indefinite preconditioner of the form:
\begin{equation}\label{BDF}
  \pd = 
\begin{bmatrix}
J_{11} & 0   \\
0 & \widehat{S}
\end{bmatrix},
\qquad
\end{equation}
and
\begin{equation}\label{CPF}
  \pc = \begin{bmatrix}
I & 0   \\
J_{12} J_{11}^{-1}  & I 
\end{bmatrix}
\begin{bmatrix}
J_{11} & 0   \\
0 & - \widehat{S} 
\end{bmatrix}
\begin{bmatrix}
I & J_{11}^{-1} J_{12}^T  \\
0 & I 
\end{bmatrix}, 
\end{equation}
where 
$$
J_{11} =  {\rm blkdiag}(M, \alpha M) \mbox{ \ and \ } J_{12} = \begin{bmatrix} \La & -\bar M \\ 0  & P_{\A}M\end{bmatrix} .
$$ 
For the reduced system in (\ref{eqN_4})  we consider 
\begin{equation}\label{BDFred}
  \pdr = 
\begin{bmatrix}
M & 0   \\
0 & \frac 1 \alpha \widehat {\mathbb S}
\end{bmatrix},
\qquad
\end{equation}
and
\begin{equation}\label{CPFred}
  \pcr = \begin{bmatrix}
I & 0   \\
\La M^{-1}  & I 
\end{bmatrix}
\begin{bmatrix}
M & 0   \\
0 & - \frac 1 \alpha \widehat {\mathbb S}
\end{bmatrix}
\begin{bmatrix}
I & M^{-1} \La^T  \\
0 & I 
\end{bmatrix}.
\end{equation}

In the following we analyze the spectral properties of the
preconditioned coefficient matrices when the above preconditioners are applied.

We recall a result from \cite{Perugia.Simoncini.00}.
\vskip 0.1in
\begin{proposition}{\rm \cite[Prop.1]{Perugia.Simoncini.00}}\label{prop:PS}
Assume the matrix ${\cal M} = [A, B^T; B, -C]$ is given, with $B$ tall, $A$ symmetric positive
definite and $C$ symmetric positive semidefinite.
Let $\gamma_{\max} = \|C\|$, $\sigma_{\max}=\|B\|$, $\alpha_0 = \lambda_{\min}(A)$ and
$\theta = \lambda_{\min}(C+BA^{-1}B^T)$. Then the eigenvalues of ${\cal M}$ are contained
in $I^-\cup I^+$ with
$$
I^- = 
\left [ \frac{-\gamma_{\max} + \alpha_0 - \sqrt{ (\gamma_{\max} + \alpha_0)^2 + 4\sigma_{\max}^2}}2,
\frac{1-\sqrt{1+4\theta}}{2}\right],
$$
$$
I^+ = \left[\alpha_0, \frac{1+\sqrt{1+4\sigma_{\max}^2}}{2}\right].
$$
\end{proposition}
\vskip 0.1in

The following result holds for the spectrum of the  matrices
$J^{aug}$ and $J^{red}$ preconditioned by the block diagonal
matrix in (\ref{BDF}) and in (\ref{BDFred}), respectively.
\vskip 0.1in
\begin{theorem}\label{thm:BDbounds}
With the notation of Proposition \ref{prop:PS},
let $\xi = \zeta^2 + (1+\zeta)^2$.

i) The eigenvalues of the block 4$\times$4 preconditioned matrix $(\pd)^{-1} {J^{aug}}$
belong to $I^-\cup I^+$ with
$I^- =\left [ \frac {1 - \sqrt{ 1+4 {\xi}}}2, \frac {1-\sqrt{3}}{2}\right ]$ and
$I^+ = [1, \frac{1+\sqrt{1+4 \xi}}{2}]$.

ii) The eigenvalues of the block 2$\times$2 preconditioned matrix $(\pdr)^{-1} {J}^{red}$
belong to $I^-\cup I^+$ with
$I^- =\left [ \frac {-\xi+1 - \sqrt{ (\xi+1)^2+4\zeta^2}}2, \frac {1-\sqrt{3}}{2}\right ]$ and
$I^+ = [1, \frac{1+\sqrt{1+4\zeta^2}}{2}]$.
\end{theorem}
\vskip 0.1in
\begin{proof}
In the proof we suppress the superscript ``BDF'' and the subscripts ``aug''/``red''.

i) We rewrite the eigenproblem $J x = \lambda {\cal P} x$ as 
 ${\cal P}^{-\frac 1 2} J{\cal P}^{-\frac 1 2}  y = \lambda y$, with
$$
{\cal P}^{-\frac 1 2} J{\cal P}^{-\frac 1 2} = 
\begin{bmatrix}
I & J_{11}^{-\frac 1 2} J_{12}^T \widehat S^{-\frac 1 2} \\
\widehat S^{-\frac 1 2} J_{12} J_{11}^{-\frac 1 2} & 0 
\end{bmatrix}.
$$
We write $S = Q {\rm blkdiag}({\mathbb S},P_{\cal A} M P_{\cal A}^T) Q^T$ and
$\widehat S = Q {\rm blkdiag}(\widehat {\mathbb S},P_{\cal A} M P_{\cal A}^T) Q^T$, with
obvious meaning for $Q$. Then we
obtain that $\widehat S^{-1} S = Q^{-T} {\rm blkdiag}(\widehat {\mathbb S}^{-1} {\mathbb S},I) Q^T $ 
so that the eigenvalues of $\widehat S^{-1} S$ are contained in $[\frac 12, \zeta^2+(1+\zeta)^2]$.
With the notation of Proposition \ref{prop:PS} and Proposition \ref{prop:upperestimate}, we have that
$\alpha_0 =1$, $\gamma_{\max}=0$, 
$\sigma_{\max}=\|\widehat S^{-1/2}S\widehat S^{-1/2}\|^{1/2} \le \sqrt{\xi}$ and
$\theta \ge \frac 12$. Therefore, substituting in the intervals of Proposition \ref{prop:PS}
we obtain that the eigenvalue $\lambda$ belongs to $I^-\cup I^+$ with
$I^- =\left [ \frac {1 - \sqrt{ 1+4 {\xi}}}2, \frac {1-\sqrt{3}}{2}\right ]$ and
$I^+ = [1, \frac{1+\sqrt{1+4 \xi}}{2}]$.

ii) We rewrite the generalized eigenproblem 
$$
\begin{pmatrix}
M  & \La^T   \\
\La & - \frac 1 {\alpha} \bar M M^{-1}\Pi_{\mathcal{I}}\bar M^T  
\end{pmatrix}
z = \lambda \begin{pmatrix}
M  &    \\
 & \frac 1 {\alpha}\widehat {\mathbb S} 
\end{pmatrix}
z
$$
as
$$
\begin{pmatrix}
I  &  {\sqrt{\alpha}} M^{-\frac 1 2}\La^T \widehat {\mathbb S}^{-\frac 1 2}  \\
\sqrt{\alpha} \widehat {\mathbb S}^{-\frac 1 2}\La M^{-\frac 1 2} & - 
\widehat {\mathbb S}^{-\frac 1 2} \bar M M^{-1}\Pi_{\mathcal{I}}\bar M^T  \widehat {\mathbb S}^{-\frac 1 2}
\end{pmatrix}
w = \lambda w .
$$
Let $\xi = \zeta^2 + (1+\zeta)^2$.
Again with the notation of Proposition \ref{prop:PS} and Proposition \ref{prop:upperestimate}, we have that
$\alpha_0 =1$, $\gamma_{\max} \le \xi$, $\sigma_{\max} = \zeta$ and
$\theta \ge \frac 1 2$. Therefore, substituting in the intervals of Proposition \ref{prop:PS}
we obtain that the eigenvalue $\lambda$ belongs to $I^-\cup I^+$ with
$I^- =\left [ \frac {-\xi+1 - \sqrt{ (\xi+1)^2+4\zeta^2}}2, \frac {1-\sqrt{3}}{2}\right ]$ and
$I^+ = [1, \frac{1+\sqrt{1+4\zeta^2}}{2}]$.
\end{proof}
\vskip 0.1in

The intervals in Theorem \ref{thm:BDbounds} have different width in the
two formulations. While the positive interval is smaller in the $2\times 2$
case, the negative one may be significantly larger, especially for large 
$\zeta$, suggesting slower convergence of the preconditioned solver. However,
we have noticed that the left extreme of $I^-$ is not very sharp (see for
instance the next example) therefore the obtained spectral intervals may 
be a little pessimistic.

In Figure \ref{fig:stime} we report
a sample of the eigenvalue estimates in Theorem \ref{thm:BDbounds} 
for the augmented and reduced systems, as the nonlinear iterations proceed. The convection-diffusion problem is considered,
with $\alpha=10^{-4}, \beta=10^{-4}$. The solid curves are the new bounds, while the circles (resp. the
asterisks) are the computed most exterior (resp. interior) eigenvalues of the preconditioned matrix.

\begin{figure}%
\centering
\subfloat[$\lambda((\pd)^{-1} J^{aug})$]
{
\includegraphics[width=.5\textwidth]{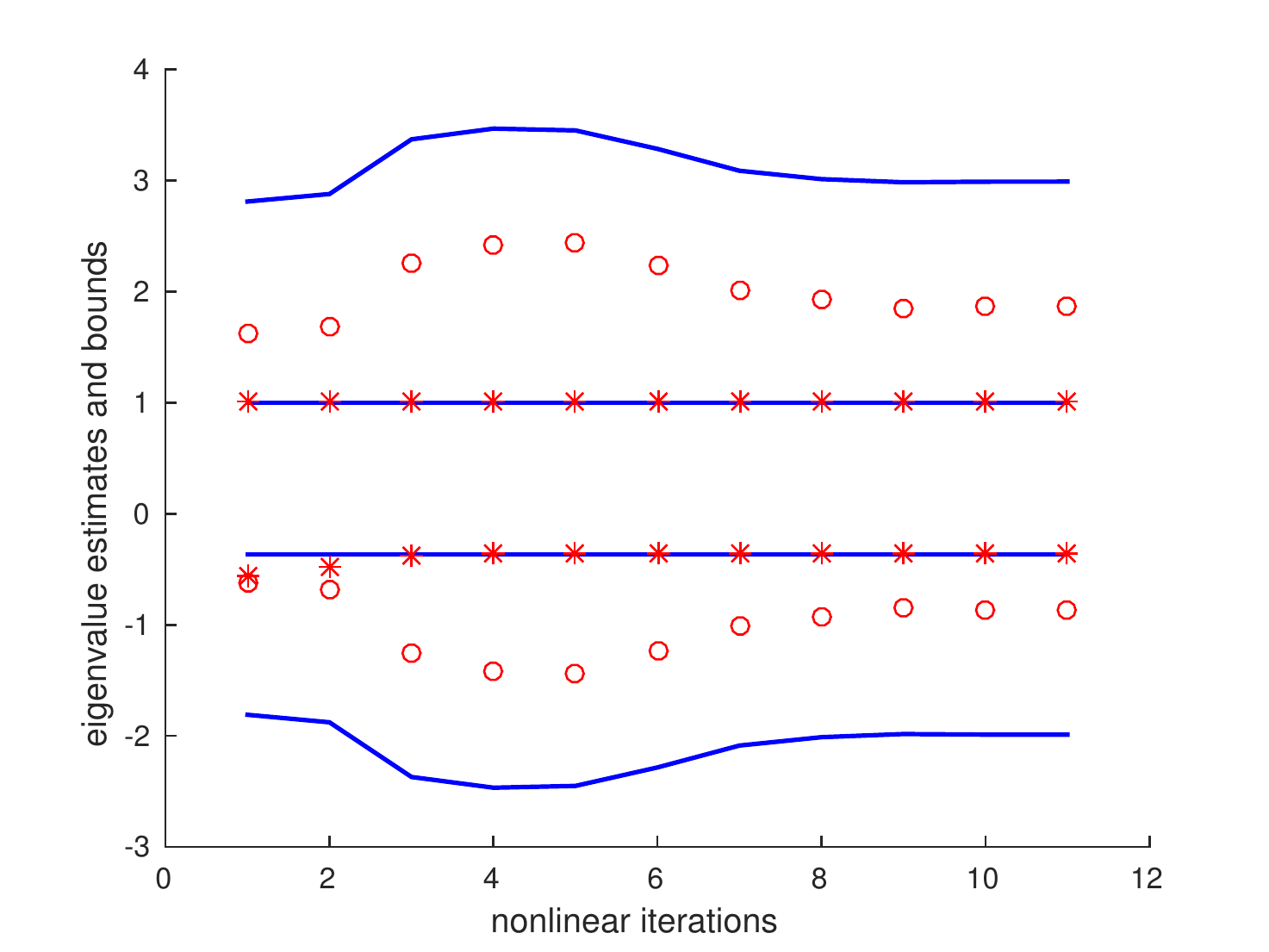}
}
\subfloat[$\lambda((\pdr)^{-1} J^{red})$]{
\includegraphics[width=.5\textwidth]{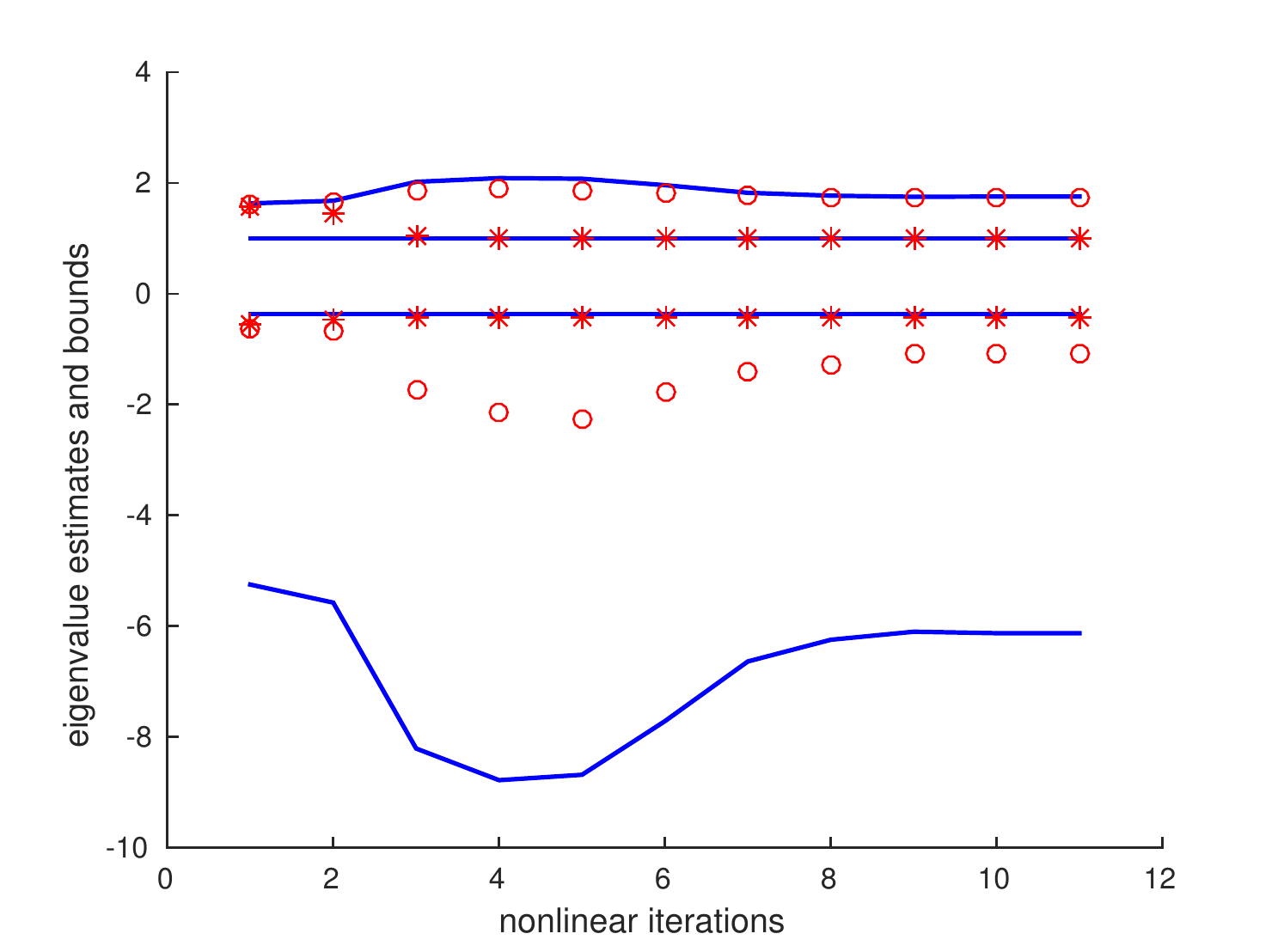}
}
\caption{Eigenvalues of the preconditioned matrices and bounds provided
 in Theorem \ref{thm:BDbounds} versus number of nonlinear iterations 
 (CD problem, $\alpha=10^{-4}, \beta=10^{-4}$). \label{fig:stime}}%
\end{figure}

The following result provides spectral information when the
indefinite preconditioner is applied to both formulations.
\vskip 0.1in
\begin{proposition}
The following results hold.

i) Let $\lambda$ be an eigenvalue of $(\pc)^{-1} J^{aug}$. Then
$\lambda \in \{1\} \cup [\frac 1 2, \zeta^2+(1+\zeta)^2]$.
Moreover, there are at least $3n+|{\cal A}|  -2|\Pi_{\cal I}|$ eigenvalues equal to $1$.
\vskip 0.08in
ii) Let $\lambda$ be an eigenvalue of $(\pcr)^{-1} J^{red}$. Then
$\lambda \in \{1\} \cup [\frac 1 2, \zeta^2+(1+\zeta)^2]$.
Moreover, there are at least $2n -2|\Pi_{\cal I}|$ eigenvalues equal to $1$.

\end{proposition}
\vskip 0.1in
\begin{proof}
i) Explicit computation shows that
$$
({\cal P}^{IPF})^{-1} J = 
\begin{bmatrix} I & J_{11}^{-1} J_{12}^T(I- \widehat S^{-1}S) \\ 0 & \widehat S^{-1}S\end{bmatrix}
$$
Since 
$$
\widehat S^{-1}S = 
\begin{bmatrix} I & 0 \\ -(\bar M \Pi_{\cal A} M^{-1}P_{\cal A}^T)^T & I \end{bmatrix}^{-1}
\begin{bmatrix} \widehat {\mathbb S}^{-1} {\mathbb S} & 0 \\ 0 & I \end{bmatrix}
\begin{bmatrix} I & 0 \\ (-\bar M \Pi_{\cal A} M^{-1}P_{\cal A}^T)^T & I \end{bmatrix} ,
$$
the eigenvalues of $\widehat S^{-1}S$ are either one, or are the eigenvalues of
$\widehat {\mathbb S}^{-1} {\mathbb S}$, which are contained in the given interval, thanks
to Proposition \ref{prop:upperestimate}.

From Proposition~\ref{prop:hatS=SpG_gen} we obtain that
$\widehat {\mathbb S}^{-1} {\mathbb S}-I$ is a low rank matrix, of rank at most $2|\Pi_{\cal I}|$.
Therefore, there are $2n+|{\cal A}|+(n-2|\Pi_{\cal I}|)$ unit eigenvalues.

ii) For the first part of the proof we proceed as above, since
$$
({\cal P}_{red}^{IPF})^{-1} J_{red} = 
\begin{bmatrix} I & M^{-1} L^T(I- \widehat S^{-1}S) \\ 0 & \widehat S^{-1}S\end{bmatrix} .
$$
For the unit eigenvalue counting,
we notice that the reduced matrix $({\cal P}_{red}^{IPF})^{-1} J_{red}$
has $n$ unit eigenvalues in the (1,1) block and 
$n-2|\Pi_{\cal I}|$ unit eigenvalues from the second block, 
for the same argument as above.
\end{proof}
\vskip 0.1in
We complete this analysis recalling that for both formulations, the
preconditioned matrix $({\cal P}^{IPF})^{-1} J$ is unsymmetric, so that
a nonsymmetric iterative solver needs to be used. In this setting, the
eigenvalues may not provide all the information required to predict the
performance of the solver; we refer the reader to \cite{Sesana.Simoncini.13} where
a more complete theoretical analysis of constraint preconditioning is proposed.

\section{Numerical experiments}\label{exp}
We implemented the semismooth Newton's method described in Algorithm \ref{algo} 
 using MATLAB\textsuperscript{\textregistered} R2016a (both linear and nonlinear solvers)
on  an Intel{\textregistered} Xeon{\textregistered}  2.30GHz, 132 GB of RAM.

Within the Newton's method we employed preconditioned \gmres\ \cite{gmres} and \minres\ \cite{minres} 
as follows 
\begin{itemize}
\item \ascp: {\sc gmres} and indefinite preconditioner $\pc/\pcr,$ 
\item \asbd: {\sc minres} and block diagonal preconditioner  $\pd/\pdr.$
\end{itemize}
The application of the Schur complement approximation $\widehat {\mathbb S}$ requires solving systems with $(\sqrt{\alpha}L + \bar M\Pi_{\I})$. As $L$ here is a discretized PDE operator such solves are quite expensive. 
We thus replace exact solves by approximate solves using an algebraic multigrid technique ({\sc hsl-mi20}) \cite{Boyle2007}. 
We use the function with a Gauss-Seidel coarse solver, $5$ steps of pre-smoothing,  and $3$ V-cycles.

The finite element matrices utilizing the SUPG technique were generated using 
the deal.II library \cite{dealii}, while standard finite differences are used for the Poisson problem.

We used a ``feasible'' starting point (see Section \ref{sec::exact}) and used 
the values $\sigma = 0.1$ and $\gamma = 10^{-4}$ in the line-search strategy described
in Algorithm \ref{algo}, as suggested in \cite{MQ95}. {We used $c=1/\alpha$ in (\ref{Fdis}) in all runs \cite{Sta09,HIK02}}.
Nonlinear iterations are stopped as soon as 
$$
\|\Theta(x_k)\| \le 10^{-6} ,
$$
and a maximum number of 100 iterations is allowed.

In order to implement the ``exact'' Semismooth Newton's method described in Section
\ref{sec::exact}, we solved the linear systems employing a strict tolerance by setting 
\begin{equation}\label{etaex}
\eta_k = 10^{-10}, \ k\ge 0.
\end{equation}

Table~\ref{tab::notation} summarizes the notation used for the numerical results. We start by illustrating 
the performance of the different formulations using (\ref{etaex}) for the Poisson problem, which is followed by the results of 
the convection-diffusion system including a discussion on inexact implementations of Algorithm \ref{algo}.

\begin{table}
\begin{center}
\begin{tabular}{ll}
\toprule
parameter& description\\
\midrule
$\ell$ & discretization level for the Poisson problem\\
 $\beta,\alpha$ & values of the regularization parameters\\
 {\sc li} &  average number of linear inner iterations\\
 {\sc nli} & number of nonlinear outer iterations\\
 {\sc cpu} &  average CPU time of the inner solver (in secs) \\
 {\sc tcpu} & total CPU time  (in secs) \\
 \%u=0 &  the percentage of zero elements in the computed control\\
\bottomrule
\end{tabular}
\caption{Legend of symbols for the numerical experiments.\label{tab::notation}}
\end{center}
\end{table}

\subsection{Poisson problem}
The first problem that we consider is the Poisson problem defined over the set
 $\Omega=(0,1)^d$, $d={2,3}$ and $\La$ is a discretization of the Laplacian from 
 (\ref{eq:lap}) using standard 
finite differences.  We set the control function bounds to $a=-30$, $b=30$
and define the desired state via
$$
y_d = \left\{\begin{array}{l l}
\sin(2\pi x)\sin(2\pi y)\exp(2x)/6 & \textnormal{ in } 2D\\
\sin(2\pi x)\sin(2\pi y)\sin(2\pi z)\exp(2x)/6 & \textnormal{ in } 3D .
\end{array} \right .
$$
In Figure~\ref{fig:dataPoi} the control $u$ for two values of $\beta$ is displayed.

\begin{figure}%
\centering
\subfloat[$\beta = 10^{-3}$]
{
\includegraphics[width=.5\textwidth]{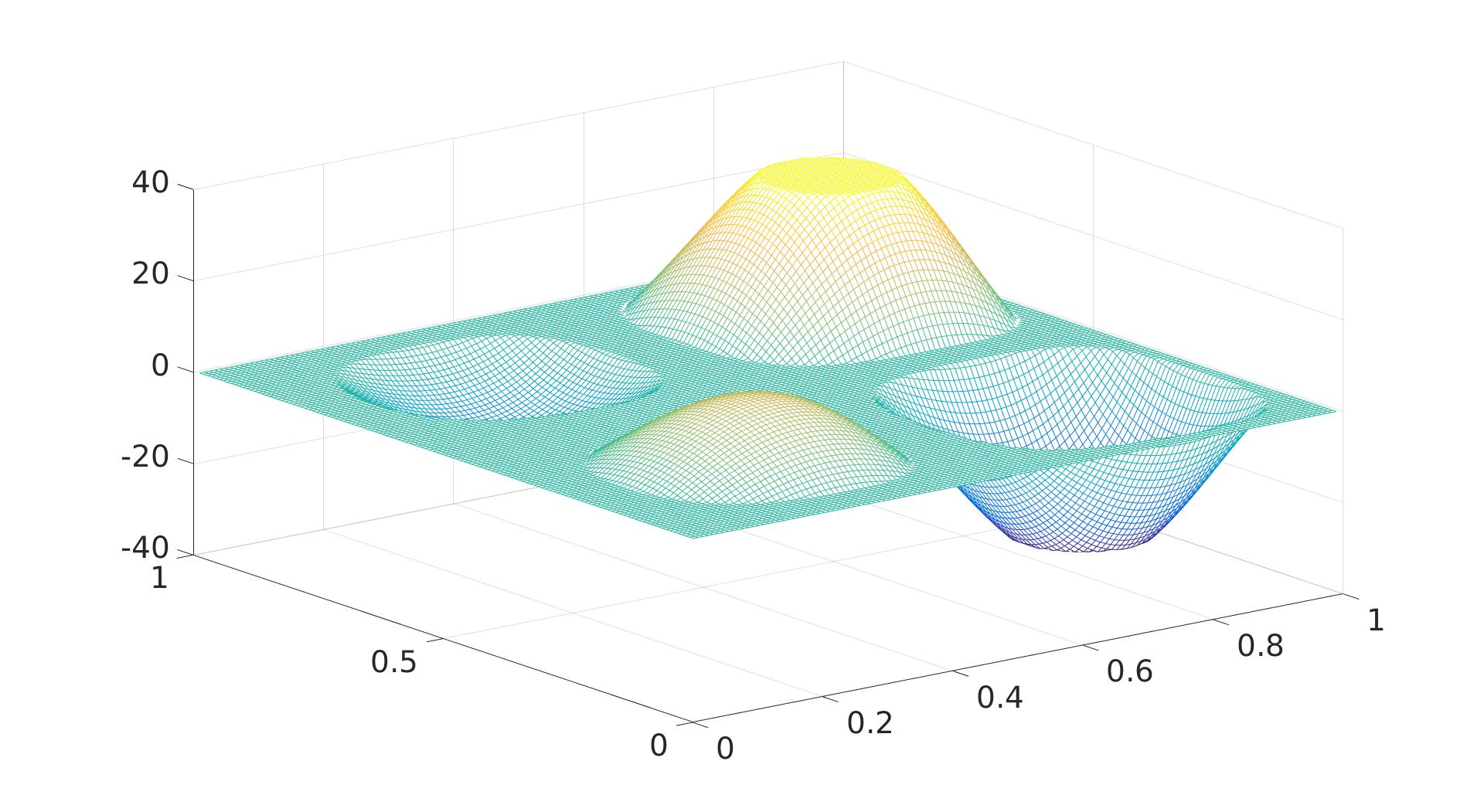}
}
\subfloat[$\beta = 10^{-4}$]{
\includegraphics[width=.5\textwidth]{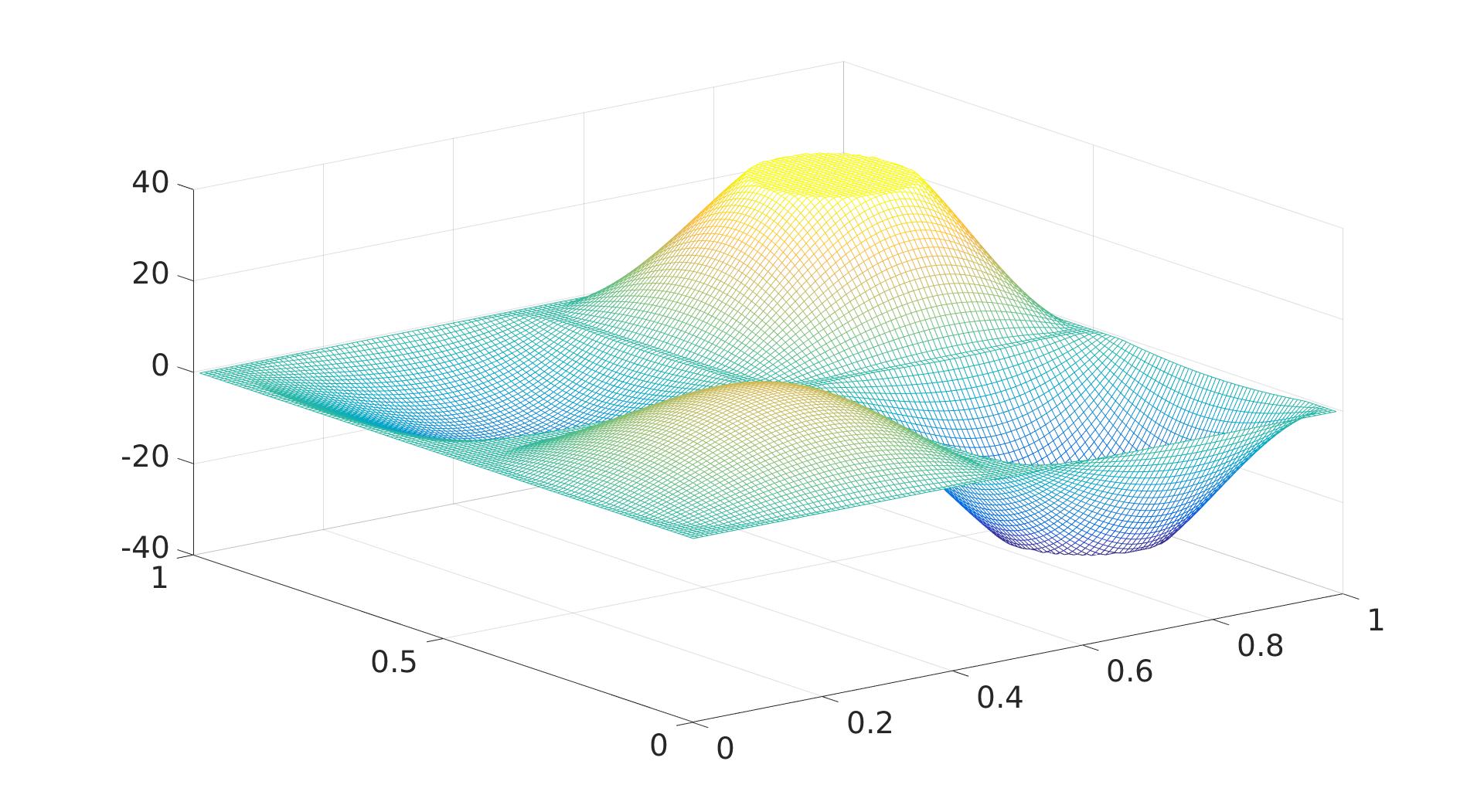}
}
\caption{Control for two different values of the parameter $\beta$. \label{fig:dataPoi}}%
\end{figure} 

As we discretize this problem with finite differences we obtain $M = \bar M = I$, the identity matrix. 
The dimension $n$ of control/state vectors is given by $n = (2^{\ell})^d$ where $\ell$ is the level
of discretization, $d=2,3$. Then, the mesh size is $h = 1/(2^{\ell}+1)$.
In the 2D case, we tested $\ell=7,8,9$ resulting in $n\in\left\lbrace16384, 65536, 262144\right\rbrace$; 
in the 3D case we set $\ell=4,5,6$ yielding $n\in\left\lbrace4096, 32768, 262144\right\rbrace$ degrees of freedom.

Figure \ref{varybeta} shows that for this problem large values of $\beta$ yield very sparse 
control $u$ and a large number of nonlinear iterations to find the solution.

\begin{figure}
\includegraphics[width= \textwidth]{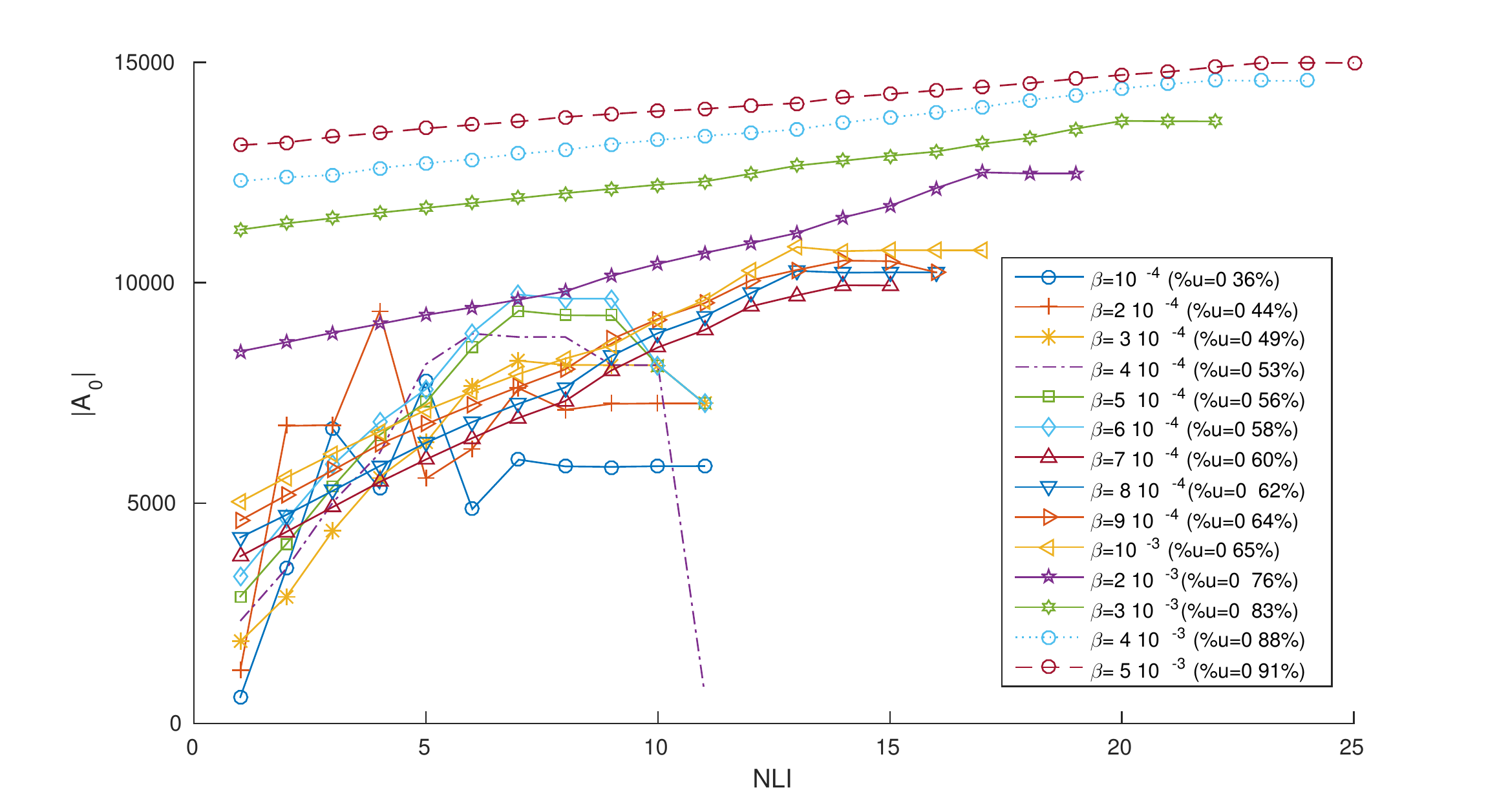}
 \caption{Number of zero components of the control $u$ versus nonlinear iterations, varying $\beta$
(parameters: $\ell = 7, \alpha= 10^{-6}$).}\label{varybeta}
\end{figure}
Table \ref{tab::resultspoisson1} shows the results for two setups applied in the case of the Poisson control problem. The results indicate that the indefinite preconditioner within {\sc gmres} performs 
remarkably better in terms of iteration numbers as well as with 
respect to computing time consumed.
Note that the number of (linear) {\sc gmres} iterations is mesh independent in both formulations,
and very mildly dependent on $\alpha$.
Moreover, this number is very low for all values of $\alpha$, implying that {\sc gmres} requires
moderate memory requirements, which can be estimated a priori by using information obtained
by running the method on a very coarse grid.
On this problem the iterative solution of the 2$\times$2 system is not significantly cheaper
than that of the 4$\times$4 one. Nonetheless, this slight improvement accumulates during
the nonlinear iterations, showing up more significantly in the total CPU time
({\sc tcpu}), for both preconditioned solvers.

\begin{table}[htb!]
\footnotesize
\begin{center}
\begin{tabular}{lllrrr|rrr|rrr}
\toprule
           &            &            &      \multicolumn{ 3}{c}{\ascp}    &     \multicolumn{ 3}{|c|}{\asbd} &&&            \\
\midrule
\multicolumn{ 1}{l}{4x4} &        $\ell$ &   $\mathrm{log}_{10}\alpha$ &  {\sc li}  &  {\sc cpu} & {\sc tcpu} &      {\sc li}  &  {\sc cpu} & {\sc tcpu} &      \%u=0 &  {\sc nli} &   {\sc bt} \\
\midrule
& \multicolumn{ 1}{l}{7} &   -2 &      11.0 &       2.0 &       4.1 &          25.0 &       6.0 &      12.1 &       3.5 &          2 &          0 \\

\multicolumn{ 1}{l}{} & \multicolumn{ 1}{l}{} &   -4 &      16.0 &       2.9 &      14.5 &           34.8 &       7.8 &      39.2 &       8.6 &          5 &          0 \\

\multicolumn{ 1}{l}{} & \multicolumn{ 1}{l}{} &   -6 &      27.1 &       6.2 &      69.1 &          61.2 &      14.0 &     154.9 &      35.6 &         11 &          9 \\
\multicolumn{ 1}{l}{} & \multicolumn{ 1}{l}{8} &   -2 &      12.0 &       7.2 &      14.4 &           27.0 &      19.2 &      38.5 &       4.3 &          2 &          0 \\

\multicolumn{ 1}{l}{} & \multicolumn{ 1}{l}{} &   -4 &      16.8 &       8.5 &      42.6 &           36.2 &      25.3&     126.9 &       9.4 &          5 &          0 \\

\multicolumn{ 1}{l}{} & \multicolumn{ 1}{l}{} &   -6 &      27.6 &      16.1 &     161.9 &          65.0 &      44.9 &     449.5 &      36.2 &         10 &          7 \\
\multicolumn{ 1}{l}{} & \multicolumn{ 1}{l}{9} &   -2 &      12.0 &      30.9 &      61.8 &          27.0 &     103.0 &     206.0 &       4.7&          2 &          0 \\

\multicolumn{ 1}{l}{} & \multicolumn{ 1}{l}{} &   -4 &      17.2 &      28.6 &     143.2 &          37.0 &      93.9 &     469.6 &       9.7 &          5 &          0 \\

\multicolumn{ 1}{l}{} & \multicolumn{ 1}{l}{} &   -6 &      28.8 &      60.2&     782.5 &          68.7 &     190.8 &    2481.5 &      36.4 &         13 &         11 \\
\multicolumn{ 1}{l}{2x2} &       & &  &&&&&&&&\\
\midrule
\multicolumn{ 1}{l}{} & \multicolumn{ 1}{l}{7} &   -2 &      11.0 &       2.0 &       4.0 &           24.0 &       4.8 &       9.6 &       3.5 &          2 &          0 \\

\multicolumn{ 1}{l}{} & \multicolumn{ 1}{l}{} &   -4 &      16.0 &       2.8 &      14.1 &           35.6 &       6.2 &      31.3 &       8.6 &          5 &          0 \\

\multicolumn{ 1}{l}{} & \multicolumn{ 1}{l}{} &   -6 &      26.7 &       5.3 &      58.9 &           65.2 &      11.2 &     124.0 &      35.6 &         11 &          9 \\

\multicolumn{ 1}{l}{} & \multicolumn{ 1}{l}{8} &   -2 &      11.5 &       4.9 &       9.9 &           25.0 &      15.6 &      31.2 &       4.3 &          2 &          0 \\

\multicolumn{ 1}{l}{} & \multicolumn{ 1}{l}{} &   -4 &      16.4 &       6.6 &      33.3 &            37.0 &      20.1 &     100.5 &       9.4 &          5 &          0 \\

\multicolumn{ 1}{l}{} & \multicolumn{ 1}{l}{} &   -6 &      27.5 &      12.1 &     121.7 &           68.3 &      36.2 &     362.6 &      36.2 &         10 &          7 \\

\multicolumn{ 1}{l}{} & \multicolumn{ 1}{l}{9} &   -2 &      12.0 &      26.0 &      52.0 &       25.5 &      95.2 &     190.5 &       4.7 &          2 &          0 \\

\multicolumn{ 1}{l}{} & \multicolumn{ 1}{l}{} &   -4 &      17.0 &      25.9 &     129.8 &          38.6 &      74.4 &     372.4 &       9.7 &          5 &          0 \\

\multicolumn{ 1}{l}{} & \multicolumn{ 1}{l}{} &   -6 &      28.5 &      50.6 &     657.7 &           71.8 &     154.4 &    2007.6 &      36.4 &         13 &         11 \\
\bottomrule
\end{tabular}  
\end{center}
\caption{Results for the Poisson problem with $\beta = 10^{-4}$ in two dimensions.
\label{tab::resultspoisson1}}
\end{table}

%
%
%
%
%
%
%
%
%

The results in Table \ref{tab::resultspoisson2} show the comparison of our proposed preconditioners for a three-dimensional 
Poisson problem. The previous results are all confirmed.
In particular,  it can again be seen that the indefinite preconditioner performs very well in comparison to 
the block-diagonal one. The timings indicate that the $2\times 2$ formulation is faster than the $4\times 4$ formulation.
The matrices stemming from the three-dimensional model are denser, therefore the gain in using the $2\times 2$ formulation
is higher. 

 \begin{table}[htb!]
 \footnotesize
\begin{center}
\begin{tabular}{lllrrr|rrr|rrr}
\toprule
           &            &            &      \multicolumn{ 3}{c}{\ascp}    &     \multicolumn{ 3}{|c|}{\asbd} &&&            \\
\midrule
\multicolumn{ 1}{l}{4x4} &        $\ell$ &   $\mathrm{log}_{10}\alpha$ &  {\sc li}  &  {\sc cpu} & {\sc tcpu} &      {\sc li}  &  {\sc cpu} & {\sc tcpu} &      \%u=0 &  {\sc nli} &   {\sc bt} \\
\midrule
\multicolumn{ 1}{l}{} & \multicolumn{ 1}{l}{4} &   -2 &      10.0 &       0.5 &       1.0 &                    21.0 &       1.9 &       3.8 &       7.4 &          2 &          0 \\

\multicolumn{ 1}{l}{} & \multicolumn{ 1}{l}{} &   -4 &      15.3 &       0.5 &       2.2 &                31.5 &       2.3 &       9.1 &       7.6 &          4 &          0 \\

\multicolumn{ 1}{l}{} & \multicolumn{ 1}{l}{} &   -6 &      24.0 &       0.9 &       7.7 &             51.1 &       3.8 &      30.6 &      38.6 &          8 &          5 \\

\multicolumn{ 1}{l}{} & \multicolumn{ 1}{l}{5} &   -2 &      10.0 &       2.5 &       5.1 &           23.0 &      11.7 &      23.5 &       8.0 &          2 &          0 \\

\multicolumn{ 1}{l}{} & \multicolumn{ 1}{l}{} &   -4 &      16.0 &       4.2 &      17.1 &      33.0 &      16.5 &      66.0 &      13.7 &          4 &          0 \\

\multicolumn{ 1}{l}{} & \multicolumn{ 1}{l}{} &   -6 &      29.7 &       7.6 &      60.7 &          62.3 &      32.6 &     261.0 &      44.4 &          8 &          2 \\

\multicolumn{ 1}{l}{} & \multicolumn{ 1}{l}{6} &   -2 &      11.0 &      19.1 &      38.2 &          23.00 &      82.21 &     164.41 &      11.8 &          2 &          0 \\

\multicolumn{ 1}{l}{} & \multicolumn{ 1}{l}{} &   -4 &      16.0 &      27.5 &     110.1 &    33.7     &      121.5 &     486.2 &      17.3 &          4 &          0 \\

\multicolumn{ 1}{l}{} & \multicolumn{ 1}{l}{} &   -6 &      31.7 &      71.3 &     570.5 &      69.1    &      261.4 &    2091.9 &      46.9 &          8 &          2 \\
\multicolumn{ 1}{l}{2x2} &       & &  &&&&&&&&\\
\midrule
\multicolumn{ 1}{l}{} & \multicolumn{ 1}{l}{4} &   -2 &      10.0 &       0.3 &       0.7 &         20.0 &       1.2 &       2.4 &       7.4 &          2 &          0 \\

\multicolumn{ 1}{l}{} & \multicolumn{ 1}{l}{} &   -4 &      14.7 &       0.5 &       2.1 &          32.7 &       1.8 &       7.4 &       7.6 &          4 &          0 \\

\multicolumn{ 1}{l}{} & \multicolumn{ 1}{l}{} &   -6 &      23.2 &       0.8 &       6.7 &         53.5 &       2.8 &      22.7 &      38.6 &          8 &          5 \\

\multicolumn{ 1}{l}{} & \multicolumn{ 1}{l}{5} &   -2 &      10.0 &       2.3 &       4.7 &           20.5 &       8.5 &      17.1 &       8.0 &          2 &          0 \\

\multicolumn{ 1}{l}{} & \multicolumn{ 1}{l}{} &   -4 &      15.5 &       3.1 &      12.7 &           33.0 &      12.9 &      51.9 &      13.7 &          4 &          0 \\
\multicolumn{ 1}{l}{} & \multicolumn{ 1}{l}{} &   -6 &      29.0 &       6.0 &      48.4 &           67.2 &      26.2 &     210.1 &      44.4 &          8 &          2 \\

\multicolumn{ 1}{l}{} & \multicolumn{ 1}{l}{6} &   -2 &      10.5 &      17.3 &      34.7 &            22.0 &      61.4 &     122.9 &      11.8 &          2 &          0 \\

\multicolumn{ 1}{l}{} & \multicolumn{ 1}{l}{} &   -4 &      16.0 &      26.1 &     104.5 &           34.7 &      92.8 &     371.5 &      17.3 &          4 &          0 \\
\multicolumn{ 1}{l}{} & \multicolumn{ 1}{l}{} &   -6 &      31.2 &      59.4 &     475.8 &          74.5 &     210.0 &    1680.3 &      46.9 &          8 &          2 \\

\bottomrule
\end{tabular}  
\end{center}\caption{Results for the Poisson problem with $\beta = 10^{-4}$ in three dimensions.
 \label{tab::resultspoisson2}}
\end{table}

\subsection{Convection diffusion problem} We consider the convection diffusion equation given by
$$\rm - \eps \Delta y + w \cdot \nabla y = u $$
with the wind $\rm w$ defined via $\rm w = (2y(1-x^2), -2x(1-y^2))$, and  the bounds $\rm a=-20$ and $\rm b = 20$. We now use finite elements in combination with the streamline upwind Galerkin (SUPG) approach for $\bar M$ and $K$ ($n \times n$), 
which are now nonsymmetric matrices. For our purposes we tested three different mesh-sizes. Namely,  $n\in\left\lbrace4225, 16641,  66049\right\rbrace$ and we also vary the influence of the diffusion by varying $\eps$, 
i.e., $\eps\in\left\lbrace 1,0.5,0.1\right\rbrace$.
{Figure~\ref{fig:dataCD} shows the control $u$ for two values of $\beta$.}
\begin{figure}%
\centering
\subfloat[$\beta = 5~10^{-2}$]{
\includegraphics[width=.4\textwidth]{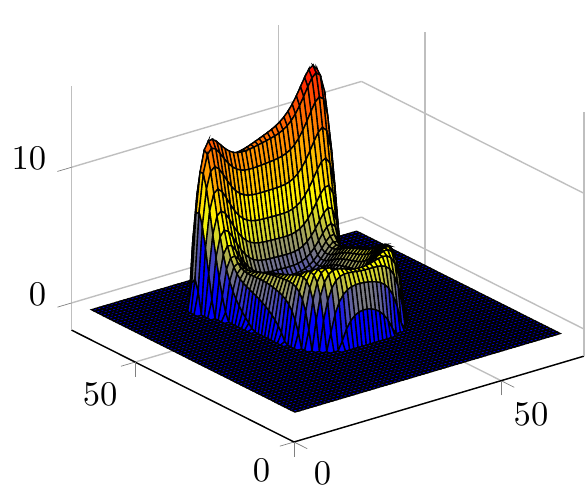}
}
\subfloat[$\beta = 10^{-2}$]{
\includegraphics[width=.4\textwidth]{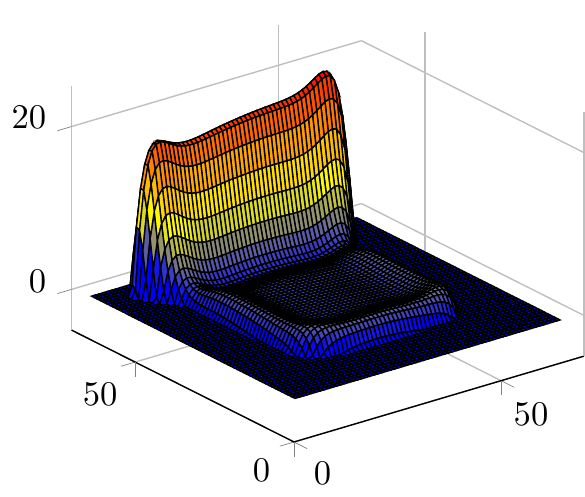}
}
\caption{Control for two different values of the parameter $\beta$. \label{fig:dataCD}}%
\end{figure} 

The results in Table \ref{tab::cd4} indicate that once more the indefinite preconditioner outperforms the block-diagonal one.
%
%
%
%
%
%
%
%
%

 \begin{table}\footnotesize
\begin{center}
\begin{tabular}{llrrr|rrr|rrr}
\toprule
     &      &      \multicolumn{ 3}{c}{\ascp}    &     \multicolumn{ 3}{|c|}{\asbd} &          \\
     \midrule
  & $\mathrm{log}_{10}\alpha$ &  {\sc li}  &  {\sc cpu} & {\sc tcpu} &        {\sc li}  &  {\sc cpu} & {\sc tcpu} &      \%u=0 &  {\sc nli} &   {\sc bt} \\
\midrule
\multicolumn{ 1}{l}{$n = 4225$} &   -1 &       10.0 &        1.2 &        2.4 &          22.0 &        1.7 &        3.5 &       12.0 &          2 &          0 \\

\multicolumn{ 1}{l}{$\eps =1$} &   -2 &       13.0 &        1.2 &        6.0 &        29.6 &        2.0 &       10.0 &       23.2 &          5 &          0 \\

\multicolumn{ 1}{l}{} &   -3 &       15.2 &        1.3 &        9.3 &             35.5 &        2.4 &       16.7 &       43.0 &          7 &          4 \\

\multicolumn{ 1}{l}{} &   -4 &       19.0 &        1.6 &       16.3 &            48.1 &        3.1 &       31.3 &       55.7 &         10 &         11 \\

\multicolumn{ 1}{l}{} &   -5 &       22.4 &        2.1 &       90.4 &             61.1 &        3.9 &      171.0 &       62.5 &         43 &        183 \\
\midrule
\multicolumn{ 1}{l}{$n = 4225$} &   -1 &       13.7 &        1.3 &        5.4 &          33.5 &        2.4 &        9.7 &       14.3 &          4 &          2 \\

\multicolumn{ 1}{l}{$ \eps =1/2$} &   -2 &       17.9 &        1.5 &       20.5 &         46.7 &        3.1 &       41.4 &       27.5 &         13 &         17 \\

\multicolumn{ 1}{l}{} &   -3 &       23.0 &        2.1 &       40.6 &            64.5 &        4.3 &       82.0 &       37.9 &         19 &         46 \\

\multicolumn{ 1}{l}{} &   -4 &       27.7 &        2.5 &       57.5 &           77.2 &        5.0 &      116.5 &       45.3 &         23 &         45 \\

\multicolumn{ 1}{l}{} &   -5 &       35.1 &        3.6 &      141.8 &           124.5 &        8.0 &      314.9 &       46.9 &         39 &        127 \\
\midrule
\multicolumn{ 1}{l}{$n = 4225$} &   -1 &       11.0 &        1.2 &        2.4 &          26.0 &        1.9 &        3.8 &       10.8 &          2 &          0 \\

\multicolumn{ 1}{l}{$\eps =1/10$} &   -2 &       14.3 &        1.2 &        7.2 &         33.8 &        2.2 &       13.3 &       24.9 &          6 &          3 \\

\multicolumn{ 1}{l}{} &   -3 &       17.8 &        1.5 &       13.5 &           44.6 &        2.8 &       25.9 &       41.1 &          9 &          8 \\

\multicolumn{ 1}{l}{} &   -4 &       22.0 &        2.0 &       28.2 &            57.7 &        3.7 &       51.7 &       51.3 &         14 &         29 \\

\multicolumn{ 1}{l}{} &   -5 &       24.2 &        2.3 &       85.1 &            69.5 &        4.4 &      158.9 &       58.8 &         36 &        156 \\
\midrule
\multicolumn{ 1}{l}{$n = 16641$} &   -1 &       13.7 &        5.0 &       20.2 &           34.2 &        9.1 &       36.4 &       12.8 &          4 &          2 \\

\multicolumn{ 1}{l}{$\eps = 1/10$} &   -2 &       17.7 &        6.4 &       77.6 &         47.1 &       12.2 &      146.8 &       25.8 &         12 &         20 \\

\multicolumn{ 1}{l}{} &   -3 &       23.2 &        8.2 &      123.6 &           62.5 &       17.1 &      257.7 &       37.2 &         15 &         25 \\

\multicolumn{ 1}{l}{} &   -4 &       29.6 &       11.2 &      224.5 &            85.3 &       21.6 &      432.9 &       45.4 &         20 &         42 \\

\multicolumn{ 1}{l}{} &   -5 &       31.5 &       15.4 &      864.3 &         100.9 &       25.9 &     1452.2 &       48.9 &         56 &        369 \\
\midrule
\multicolumn{ 1}{l}{$n = 66049$} &   -1 &       13.7 &       17.7 &       70.8 &         34.2 &       33.7 &      134.8 &       12.1 &          4 &          2 \\

\multicolumn{ 1}{l}{$\eps = 1/10$} &   -2 &       18.7 &       22.7 &      295.0 &           50.6 &       48.5 &      630.5 &       25.3 &         13 &         18 \\

\multicolumn{ 1}{l}{} &   -3 &       23.4 &       27.0 &      406.0 &             63.4 &       55.2 &      828.0 &       36.6 &         15 &         24 \\

\multicolumn{ 1}{l}{} &   -4 &       32.4 &       39.7 &     1231.8 &           103.3 &       87.8 &     2722.4 &       45.2 &         31 &         84 \\

\multicolumn{ 1}{l}{} &   -5 &       30.6 &       42.1 &     4004.2 &            98.6 &       84.3 &     8010.8 &       48.9 &         95 &        783 \\
\bottomrule
\end{tabular}  
\end{center}
\caption{Convection-Diffusion problem: comparison between \ascp and \asbd using the $2\times 2$ formulation for
various settings
(parameters: $\beta= 10^{-2}$).}
\label{CD3}
\end{table}

A comparison of both formulations with respect to changes in the mesh-size is shown in Table \ref{tab::cd4}. We see that the performance of the iterative solver for the linear system is robust with respect to changes in the mesh-size but also with respect to the two formulations presented.
We also remark that as $\alpha$ gets smaller, the number of back tracking ({\sc bt}) iterations increases showing that the
problem is much harder to solve, and the line-search strategy regularizes the Newton's model
by damping the step. Once again, the reduced formulation is more competitive than the original one.

\begin{table}[htb!]
\footnotesize
\begin{center}
\begin{tabular}{llrrrrrr}
\toprule
     &      &      \multicolumn{ 6}{c}{\ascp}              \\
     \midrule
$n = 16661$ &   $\mathrm{log}_{10}\alpha$ &  {\sc li}  &  {\sc cpu} & {\sc tcpu} &      \%u=0 &  {\sc nli} &   {\sc bt} \\
\midrule
\multicolumn{ 1}{l}{$4 \times 4$} &   -1 &      13.75 &       5.38 &      21.52 &      12.85 &          4 &          2 \\

\multicolumn{ 1}{l}{} &   -2 &      17.75 &       6.68 &      80.18 &      25.83 &         12 &         20 \\

\multicolumn{ 1}{l}{} &   -3 &      23.27 &       9.16 &     137.35 &      37.29 &         15 &         25 \\

\multicolumn{ 1}{l}{} &   -4 &      30.25 &      13.07 &     261.45 &      45.45 &         20 &         42 \\

\multicolumn{ 1}{l}{} &   -5 &      33.09 &      17.33 &     970.60 &      48.98 &         56 &        369  \\
\midrule
\multicolumn{ 1}{l}{$2\times 2$} &   -1 &       13.7 &        5.0 &       20.2 &       12.8 &          4 &          2 \\

\multicolumn{ 1}{l}{} &   -2 &       17.7 &        6.4 &       77.6 &       25.8 &         12 &         20 \\

\multicolumn{ 1}{l}{} &   -3 &       23.2 &        8.2 &      123.6 &       37.2 &         15 &         25 \\

\multicolumn{ 1}{l}{} &   -4 &       29.6 &       11.2 &      224.5 &       45.4 &         20 &         42 \\

\multicolumn{ 1}{l}{} &   -5 &       31.5 &       15.4 &      864.3 &       48.9 &         56 &        369 \\
%
\midrule
$n= 66049$ &   $\mathrm{log}_{10}\alpha$ &  {\sc li}  &  {\sc cpu} & {\sc tcpu} &      \%u=0 &  {\sc nli} &   {\sc bt}             \\
\midrule
\multicolumn{ 1}{l}{$4 \times 4$} &   -1 &      13.75 &      18.97 &      75.87 &      12.16 &          4 &          2             \\

\multicolumn{ 1}{l}{} &   -2 &      18.77 &      26.23 &     340.99 &      25.39 &         13 &         18             \\

\multicolumn{ 1}{l}{} &   -3 &      23.47 &      33.50 &     502.43 &      36.64 &         15 &         24             \\

\multicolumn{ 1}{l}{} &   -4 &      32.48 &      48.09 &    1490.65 &      45.27 &         31 &         84           \\

\multicolumn{ 1}{l}{} &   -5 &      31.31 &      53.23 &    5056.79 &      48.98 &         95 &        783             \\
\midrule
\multicolumn{ 1}{l}{$2\times 2$} &   -1 &       13.7 &       17.7 &       70.8 &       12.1 &          4 &          2             \\

\multicolumn{ 1}{l}{} &   -2 &       18.7 &       22.7 &      295.0 &       25.3 &         13 &         18            \\

\multicolumn{ 1}{l}{} &   -3 &       23.4 &       27.0 &      406.0 &       36.6 &         15 &         24             \\

\multicolumn{ 1}{l}{} &   -4 &       32.4 &       39.7 &     1231.8 &       45.2 &         31 &         84             \\

\multicolumn{ 1}{l}{} &   -5 &       30.6 &       42.1 &     4004.2 &       48.9 &         95 &        783           \\
\bottomrule
\end{tabular}  
\end{center}\caption{Convection-Diffusion problem: comparison between original and reduced
 formulations using \ascp (parameters:  $\beta= 10^{-2}, \eps =1/10$).\label{tab::cd4}}
\end{table}

\vskip 0.1in
We conclude this section discussing the inexact implementation of the Newton's method.
Following the results by Eisenstat and Walker \cite{EWchoosing} for smooth equations,
we chose the so-called  adaptive $Choice$ 2 for the forcing term $\eta_k$ in (\ref{res}) in order to achieve the desirable fast local convergence
near a solution and, at the same time, to minimize the oversolving: we set
\begin{equation}\label{ls}
\eta_k=\chi \left(\frac{\|\Theta_{k+1}\|_2}{\|\Theta_k\|_2}\right)^2, \, \,
  k\ge 1,
\end{equation}
with $\chi=0.9$ and safeguard  
$$
\eta_k=\max\{\eta_k,\chi \eta_{k-1}^{2}\},
$$
if $\chi \eta_{k-1}^{2}>0.1$; then,  the additional safeguard
$\eta_k=\min\{\eta_k,\eta_{max}\}$ is used. 
We considered the values $\eta_0=\eta_{max} \in \{10^{-1}, 10^{-2}, 10^{-2}, 10^{-10}\}$
to explore the impact of the linear solver accuracy on the overall Newton's performance.

In Figure \ref{fig:eta}, we plot the overall CPU time and the average number of linear iterations
varying $\eta_0$ for the convection-diffusion problem. {These results 
were obtained using the reduced formulation of the Newton's 
equation with the residual test in (\ref{eqN_2_inex}).
Nevertheless, we remark that 
we obtained similar results using the augmented formulation (\ref{eqN_4_inex}), that is the same number 
of {\sc NLI} and {\sc LI} but clearly different values for {\sc tcpu}.

We note that the gain in CPU time increases for looser accuracy due to
the decrease in the linear iterations (almost constant with size). In particular, on average,
$\eta_0=10^{-1}$ yields a gain of the 65\% of {\sc tcpu} with respect to the ``exact'' choice 
$\eta_0=10^{-10}$ while the gain with $\eta_0=10^{-4}$ is of the 46\%.

}

\begin{figure}%
\centering
\subfloat[{\sc tcpu} versus $n$]{
\includegraphics[width=.45\textwidth]{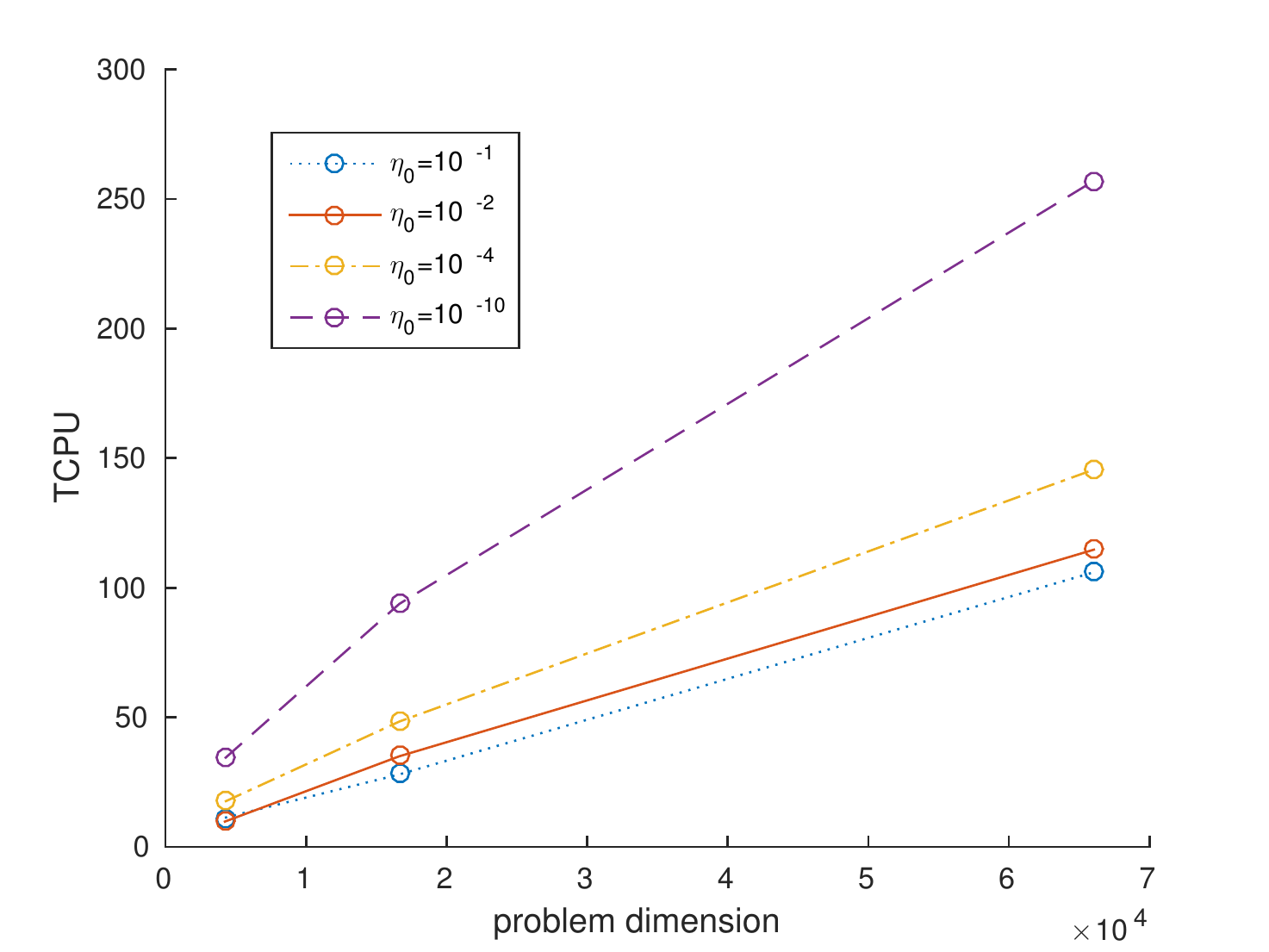}
}
\subfloat[{\sc li} versus $n$]{
\includegraphics[width=.45\textwidth]{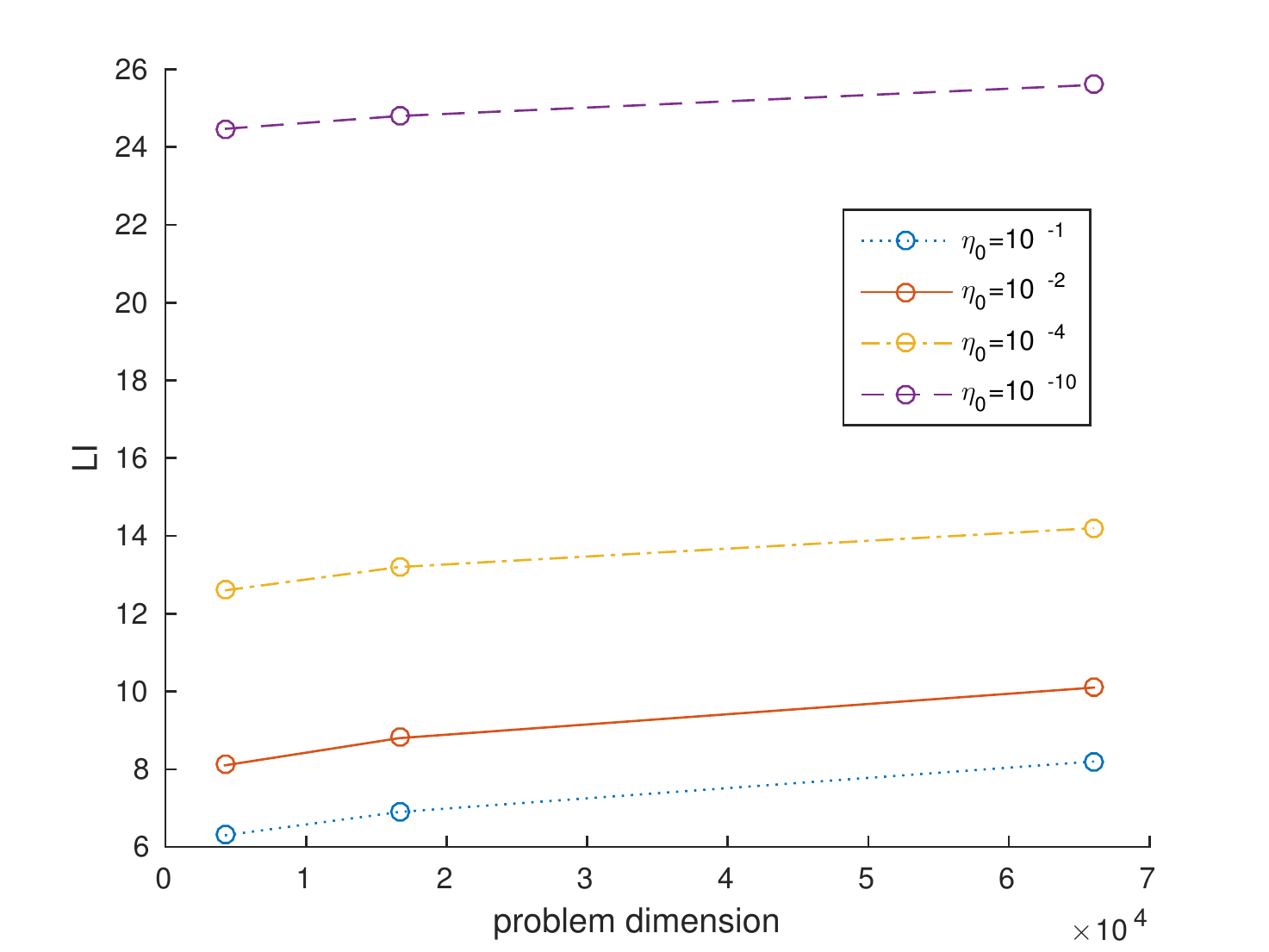}
}
\caption{ Convection-diffusion problem.
Behavior of the inexact Newton's method for different values of the parameter $\eta_0$ (parameters: $\alpha = 10^{-3}, \beta= 10^{-2}, \eps = 1/10$).
\label{fig:eta}}%
\end{figure}

\section{Conclusions}
{We have presented a general semismooth Newton's algorithm for
the solution of bound constrained optimal control problems where a sparse control is sought.
On the one side we have analyzed the nonlinear scheme in the framework 
of global convergent inexact semismooth Newton methods;
on the other side, we have enhanced the solution of the linear algebra phase by proposing reduced formulation of the Newton's equation
and preconditioners based on the active-set Schur complement approximations. We have provided a theoretical
support of the proposed techniques and validated the proposals on large scale Poisson and convection-diffusion problems.}

\section*{Acknowledgements}
Part of the work of the first two authors was supported by INdAM-GNCS, Italy, under the
2016 Project  \emph{Equazioni e funzioni di matrici con struttura: analisi e algoritmi}.

\bibliographystyle{siam}
\bibliography{data}

\begin{thebibliography}{10}

\bibitem{axelsson2015comparison}
{\sc O.~Axelsson, S.~Farouq, and M.~Neytcheva}, {\em Comparison of
  preconditioned {K}rylov subspace iteration methods for {PDE}-constrained
  optimization problems}, Numer. Algorithms, 73 (2016), pp.~631--663.

\bibitem{dealii}
{\sc W.~Bangerth, R.~Hartmann, and G.~Kanschat}, {\em deal.{II}---a
  general-purpose object-oriented finite element library}, ACM Trans. Math.
  Software, 33 (2007), pp.~Art. 24, 27.

\bibitem{BenGolLie05}
{\sc M.~Benzi, G.~Golub, and J.~Liesen}, {\em Numerical solution of saddle
  point problems}, Acta Numer, 14 (2005), pp.~1--137.

\bibitem{BIK99}
{\sc M.~Bergounioux, K.~Ito, and K.~Kunisch}, {\em Primal-dual strategy for
  constrained optimal control problems}, SIAM J. Control Optim., 37 (1999),
  pp.~1176--1194.

\bibitem{Boyle2007}
{\sc J.~Boyle, M.~D. Mihajlovi\'c, and J.~Scott}, {\em {HSL\_MI20: an efficient
  AMG preconditioner for finite element problems in 3D}}, Int. J. Numer. Meth.
  Engnrg,, 82 (2010), pp.~64--98.

\bibitem{BroH82}
{\sc A.~Brooks and T.~Hughes}, {\em Streamline upwind/{Petrov-Galerkin}
  formulations for convection dominated flows with particular emphasis on the
  incompressible {Navier-Stokes} equations}, Computer methods in applied
  mechanics and engineering, 32 (1982), pp.~199--259.

\bibitem{CanW08}
{\sc E.~J. Cand{\`e}s and M.~B. Wakin}, {\em An introduction to compressive
  sampling}, IEEE Signal Processing Magazine, 25 (2008), pp.~21--30.

\bibitem{Don06}
{\sc D.~L. Donoho}, {\em Compressed sensing}, IEEE Trans. Inform. Theory, 52
  (2006), pp.~1289--1306.

\bibitem{EWchoosing}
{\sc S.~Eisenstat and H.~Walker}, {\em {Choosing the forcing terms in an
  inexact Newton method}}, SIAM Journal on Scientific Computing, 17 (1996),
  pp.~16--32.

\bibitem{ElmSW14}
{\sc H.~Elman, D.~Silvester, and A.~Wathen}, {\em Finite elements and fast
  iterative solvers: with applications in incompressible fluid dynamics},
  Oxford University Press, 2014.

\bibitem{GanS12}
{\sc S.~Ganguli and H.~Sompolinsky}, {\em Compressed sensing, sparsity, and
  dimensionality in neuronal information processing and data analysis}, Annual
  Review of Neuroscience, 35 (2012), pp.~485--508.

\bibitem{GunHS}
{\sc A.~G{\"u}nnel, R.~Herzog, and E.~Sachs}, {\em {A note on preconditioners
  and scalar products for {K}rylov methods in {H}ilbert space}}, Electronic
  Transactions on Numerical Analysis, 40 (2014), pp.~13--20.

\bibitem{HeiL10}
{\sc M.~Heinkenschloss and D.~Leykekhman}, {\em Local error estimates for
  {SUPG} solutions of advection-dominated elliptic linear-quadratic optimal
  control problems}, SIAM Journal on Numerical Analysis, 47 (2010),
  pp.~4607--4638.

\bibitem{HerOW15}
{\sc R.~Herzog, J.~Obermeier, and G.~Wachsmuth}, {\em Annular and sectorial
  sparsity in optimal control of elliptic equations}, Computational
  Optimization and Applications, 62 (2015), pp.~157--180.

\bibitem{HS10}
{\sc R.~Herzog and E.~Sachs}, {\em Preconditioned conjugate gradient method for
  optimal control problems with control and state constraints}, SIAM J. Matrix
  Anal. Appl., 31 (2010), pp.~2291--2317.

\bibitem{HSW11_DS}
{\sc R.~Herzog, G.~Stadler, and G.~Wachsmuth}, {\em {Directional sparsity in
  optimal control of partial differential equations.}}, {SIAM J. Control
  Optim.}, 50 (2012), pp.~943--963.

\bibitem{HIK02}
{\sc M.~Hinterm{\"u}ller, K.~Ito, and K.~Kunisch}, {\em The primal-dual active
  set strategy as a semismooth {N}ewton method}, SIAM J. Optim., 13 (2002),
  pp.~865--888.

\bibitem{book::hpuu09}
{\sc M.~Hinze, R.~Pinnau, M.~Ulbrich, and S.~Ulbrich}, {\em Optimization with
  {PDE} Constraints}, Mathematical Modelling: Theory and Applications,
  Springer-Verlag, New York, 2009.

\bibitem{book::IK08}
{\sc K.~Ito and K.~Kunisch}, {\em Lagrange multiplier approach to variational
  problems and applications}, vol.~15 of Advances in Design and Control,
  Society for Industrial and Applied Mathematics (SIAM), Philadelphia, PA,
  2008.

\bibitem{MW10}
{\sc K.~Mardal and R.~Winther}, {\em {Construction of preconditioners by
  mapping properties}}, Bentham Science Publishers, 2010, ch.~4, pp.~65--84.

\bibitem{MQ95}
{\sc J.~Mart{\'\i}nez and L.~Qi}, {\em Inexact newton methods for solving
  nonsmooth equations}, Journal of Computational and Applied Mathematics, 60
  (1995), pp.~127--145.

\bibitem{Nak16}
{\sc J.~Nakamura}, {\em Image sensors and signal processing for digital still
  cameras}, CRC press, 2016.

\bibitem{NocW06}
{\sc J.~Nocedal and S.~Wright}, {\em Numerical optimization}, Springer Series
  in Operations Research and Financial Engineering, Springer, New York,
  second~ed., 2006.

\bibitem{minres}
{\sc C.~C. Paige and M.~A. Saunders}, {\em Solutions of sparse indefinite
  systems of linear equations}, SIAM J. Numer. Anal, 12 (1975), pp.~617--629.

\bibitem{PW10}
{\sc J.~W. Pearson and A.~Wathen}, {\em {A new approximation of the Schur
  complement in preconditioners for PDE-constrained optimization}}, Numerical
  Linear Algebra with Applications, 19 (2012), pp.~816--829.

\bibitem{PW11}
{\sc J.~W. Pearson and A.~J. Wathen}, {\em Fast iterative solvers for
  convection-diffusion control problems}, Electronic Transactions on Numerical
  Analysis, 40 (2013), pp.~294--310.

\bibitem{Perugia.Simoncini.00}
{\sc I.~Perugia and V.~Simoncini}, {\em Block--diagonal and indefinite
  symmetric preconditioners for mixed finite element formulations}, Numerical
  Linear Algebra with Applications, 7 (2000), pp.~585--616.

\bibitem{pst15}
{\sc M.~Porcelli, V.~Simoncini, and M.~Tani}, {\em Preconditioning of
  active-set {N}ewton methods for {PDE}-constrained optimal control problems},
  SIAM Journal on Scientific Computing, 37 (2015), pp.~S472--S502.

\bibitem{Ree10}
{\sc T.~Rees}, {\em Preconditioning Iterative Methods for PDE Constrained
  Optimazation}, PhD thesis, University of Oxford, 2010.

\bibitem{book::saad}
{\sc Y.~Saad}, {\em Iterative methods for sparse linear systems}, Society for
  Industrial and Applied Mathematics, Philadelphia, PA, 2003.

\bibitem{gmres}
{\sc Y.~Saad and M.~Schultz}, {\em G{MRES}: a generalized minimal residual
  algorithm for solving nonsymmetric linear systems}, SIAM J. Sci. Statist.
  Comput, 7 (1986), pp.~856--869.

\bibitem{SchU12}
{\sc A.~Schiela and S.~Ulbrich}, {\em {Operator Preconditioning for a Class of
  Constrained Optimal Control Problems}}, SIAM J. Optim., 24 (2012),
  pp.~435--466.

\bibitem{Sesana.Simoncini.13}
{\sc D.~Sesana and V.~Simoncini}, {\em Spectral analysis of inexact constraint
  preconditioning for symmetric saddle point matrices}, Linear Algebra and its
  Applications, 438 (2013), pp.~2683--2700.

\bibitem{Sta09}
{\sc G.~Stadler}, {\em {Elliptic optimal control problems with $L^1$-control
  cost and applications for the placement of control devices.}}, {Comput.
  Optim. Appl.}, 44 (2009), pp.~159--181.

\bibitem{SPW10}
{\sc M.~Stoll, J.~W. Pearson, and A.~Wathen}, {\em {Preconditioners for state
  constrained optimal control problems with Moreau-Yosida penalty function}},
  Numer. Lin. Alg. Appl., 21 (2014), pp.~81--97.

\bibitem{Sun10}
{\sc T.~Sun}, {\em Discontinuous galerkin finite element method with interior
  penalties for convection diffusion optimal control problem}, Int. J. Numer.
  Anal. Mod, 7 (2010), pp.~87--107.

\bibitem{book::FT2010}
{\sc F.~Tr{\"o}ltzsch}, {\em {Optimal Control of Partial Differential
  Equations: Theory, Methods and Applications}}, Amer Mathematical Society,
  2010.

\bibitem{U11}
{\sc M.~Ulbrich}, {\em {Semismooth Newton Methods for Variational Inequalities
  and Constrained Optimization Problems}}, SIAM Philadelphia, 2011.

\end{thebibliography}
\end{document}